\documentclass[a4paper,twocolumn]{article}
\usepackage[top=1.75cm, bottom=1.75cm, left=1.25cm, right=1.25cm]{geometry}
\usepackage[colorlinks=true,breaklinks=true,bookmarks=true,urlcolor=blue,
     citecolor=blue,linkcolor=blue,bookmarksopen=false,draft=false]{hyperref}

\usepackage[utf8]{inputenc}
\renewcommand{\l}{\left}
\renewcommand{\r}{\right}

\usepackage{microtype}
\usepackage{enumerate}
\usepackage{amsmath}
\usepackage{amssymb}
\usepackage[square,numbers]{natbib}
\usepackage{amsthm}
\usepackage{tabularx} 
\usepackage{booktabs}
\usepackage{xspace}
\usepackage{comment}
\usepackage{algorithm}
\usepackage{algorithmic}
\usepackage{eufrak}
\usepackage{color}
\definecolor{red}{rgb}{1.0,0.0,0.0}
\definecolor{blue}{rgb}{0.0,0.0,1.0}
\newcommand{\niek}[1]{\textcolor{red}{\textsf{Niek: #1}}\xspace}
\newcommand{\andy}[1]{\textcolor{blue}{\textsf{Andrey: #1}}\xspace}

\usepackage[font=small,skip=6pt]{caption}
\usepackage[font=small]{subcaption}
\DeclareMathOperator*{\argmin}{arg\,min}

\newcommand{\reflem}[1]{Lemma~\ref{lem:#1}\xspace}
\newcommand{\refthm}[1]{Theorem~\ref{thm:#1}\xspace}
\newcommand{\refdef}[1]{Definition~\ref{def:#1}\xspace}

\newcommand{\refsec}[1]{Section~\ref{sec:#1}\xspace}

\newcommand{\reffig}[1]{Figure~\ref{fig:#1}\xspace}
\newcommand{\refprop}[1]{Proposition~\ref{prop:#1}\xspace}

\newcommand{\term}[1]{\textsl{#1}\xspace}

\newcommand{\posint}{\ensuremath{\mathbb{N}_{>0}}\xspace}
\newcommand{\mcal}[1]{\ensuremath{\mathcal{#1}}\xspace}

\newtheorem{theorem}{Theorem}
\newtheorem{lemma}{Lemma}
\newtheorem{corollary}{Corollary}
\newtheorem{proposition}{Proposition}
\newtheorem{observation}{Observation}
\newtheorem{remark}{Remark}

\newtheorem{definition}{Definition}
\newtheorem{assumption}{Assumption}

\newcommand{\diam}{\mathsf{diam}}

\newcommand{\cor}{\mathsf{cor}}

\newcommand{\inn}{\mathsf{inn}}

\newcommand{\vsect}[2][c]{#2_{\{#1\}}} 
\newcommand{\conv}{\mathsf{ch}\,}

\newcommand{\vor}[2][]{\mathrm{Vor}_{#1}(#2)} 

\newcommand{\natnum}{\mathbb{N}}
\newcommand{\reals}{\mathbb{R}}

\renewcommand{\S}{\ensuremath{\mathcal{S}}\xspace}
\renewcommand{\P}{\ensuremath{\conv{\S}}\xspace}

\renewcommand{\SS}{\ensuremath{\mathbb{S}}\xspace}

\title{Average Point Pursuit using the Greedy Algorithm:\\Theory and Applications}
\author{Andrey Bernstein\\National Renewable Energy Laboratory\\Golden, CO, USA \\ \href{mailto:andrey.bernstein@nrel.gov}{\texttt{andrey.bernstein@nrel.gov}} \and Niek J. Bouman\\Dept.\ of Mathematics and Computer Science\\TU Eindhoven, the Netherlands\\\href{mailto:n.j.bouman@tue.nl}{\texttt{n.j.bouman@tue.nl}}}

\date{}

\usepackage{cryptocode}

\begin{document}

  \setlength{\parskip}{0pt}
\bibliographystyle{abbrvnat}

\maketitle

\begin{abstract}

This paper considers a discrete-time decision problem wherein a decision maker has to track, on average, a sequence of inputs 
selected from a convex set $\mathcal X \subset \reals^d$ 
by choosing actions 
from a possibly non-convex feasible set $\mathcal Y\subset \reals^d$, 
where $\mathcal X$ is in fact the convex hull of $\mathcal Y$.
We study some generalized variants of this problem, in which: (i) $\mathcal X$ and $\mathcal Y$ vary with time, and (ii) there might be a \emph{delay} between them, in the sense that $\mathcal X$ is the convex hull of the \emph{previous} $\mathcal Y$.
We investigate the conditions under which the greedy algorithm that minimizes, in an online fashion, the accumulated error between the sequence of inputs and decisions, is able to track the average input asymptotically. 
Essentially, this comes down to proving that the accumulated error, whose evolution is governed by a non-linear dynamical system, remains within a bounded invariant set.
%
Applications include control of discrete devices using continuous setpoints; control of highly uncertain devices with some information delay; and digital printing, scheduling, and assignment problems.
\end{abstract}

\paragraph{Keywords:} Convex dynamics, stability theory of non-linear systems, error diffusion, average tracking
\vspace{.5em}

\section{Introduction}
We consider a discrete-time decision problem, wherein at each time step $n = 0, 1, \ldots$, a decision maker has to choose an action $y_n$ that lies in a (possibly) \emph{non-convex} set $\S_n \in \mathbb S$, where $\mathbb S$ is a  possibly infinite collection of subsets of  
$\reals^d$. The goal of the decision maker is to track \emph{on average} an arbitrary sequence of input variables $x_0, x_1, \ldots$, with $x_n \in \conv \S_n$, where $\conv (\cdot)$ denotes the convex hull operation. That is, the decisions $\{y_n\}$ have to satisfy
\begin{equation} \label{eqn:average_error}
\l \| \frac{1}{n}\sum_{i = 0}^{n-1} x_i - \frac{1}{n}\sum_{i = 0}^{n-1} y_i \r \| \rightarrow 0
\end{equation}
as $n \rightarrow \infty$. In this paper, we investigate conditions under which the \emph{greedy algorithm} that minimizes the total accumulated error
\begin{equation} \label{eqn:e_n}
e_n := \sum_{i = 0}^{n-1} (x_i - y_i), \quad n \geq 1
\end{equation}
attains \eqref{eqn:average_error}. In particular, our focus is on the algorithm that chooses $y_n$ that minimizes 
$
\l \|e_{n} + x_n - y \r\|
$
over $y \in \S_n$, for all $n \geq 0$ (with $e_0 := \boldsymbol{0}$).  

It is convenient to view this problem as a game between three parties: (i) the decision maker (i.e., the greedy algorithm), (ii) an opponent, and (iii) an oracle. At each time step:
\begin{enumerate}
\setlength{\itemsep}{0pt}%
\setlength{\parskip}{0pt}
\item[(a)] The oracle specifies the feasible set $\S \in \SS$ and broadcasts it to the decision maker and the opponent;
\item[(b)] The opponent chooses $x \in \conv \S$ and sends it to the decision maker; 
\item[(c)] The decision maker chooses $y \in \S$ according the greedy algorithm.
\end{enumerate}
We distinguish between two cases, as illustrated in Figures \ref{problem1} and \ref{problem2}. Figure \ref{problem1} specifies the case with the following order of operations at each time step: $(a) \rightarrow (b) \rightarrow (c)$; while Figure \ref{problem2} specifies the case where the order is $(b) \rightarrow (a) \rightarrow (c)$. Note that in the latter case, there is a delay of information on the opponent side: the opponent chooses the input $x$ based on a delayed (previous) instance of the set $\S$. We thus term these two problems as \emph{undelayed} and \emph{delayed}, respectively.

\begin{figure}
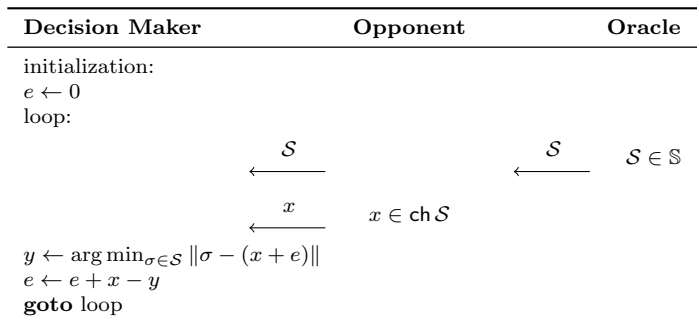

  \footnotesize
\begin{tabularx}{\columnwidth}{lXcXr}
\toprule 
\textbf{Decision Maker} & & \textbf{Opponent} && \textbf{Oracle}\\
\midrule  
initialization:\\
$e \gets 0$ \\
loop:\\
&$\sendmessageleft*[1cm]{\S}$&&$\sendmessageleft*[1cm]{\S}$& $\mathcal S \in \mathbb S$\\
&$\sendmessageleft*[1cm]{x}$& $x \in \conv \mathcal S$\\
\multicolumn{2}{l}{$y \gets \arg\min_{\sigma \in \mathcal{S}} \| \sigma - (x+e)\|  $}\\
$e \gets e + x - y $\\
\textbf{goto} loop\\
\bottomrule 
\end{tabularx}
\caption{The \textsl{Undelayed Problem}, in which the opponent chooses the input $x$ from the (convex hull of the) \emph{same} set $\mathcal S$ from which the decision maker chooses her action. \label{problem1}}
\end{figure}

\begin{figure}
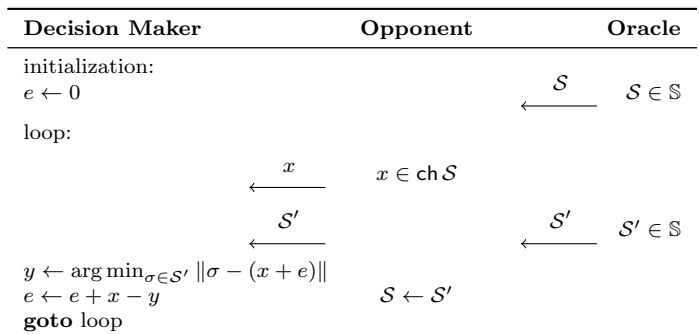

  \footnotesize
\begin{tabularx}{\columnwidth}{lXcXr}
\toprule 
\textbf{Decision Maker} & & \textbf{Opponent} && \textbf{Oracle}\\
\midrule  
initialization:\vspace{-.7em}\\
$e \gets 0$ 
&&&  $\sendmessageleft*[1cm]{\mathcal S}$ & $\mathcal S \in \mathbb S$ \\
loop:\\
&$\sendmessageleft*[1cm]{x}$& $x \in \conv \mathcal S$\\
&$\sendmessageleft*[1cm]{\S'}$&&$\sendmessageleft*[1cm]{\S'}$& $\mathcal S' \in \mathbb S$\\
\multicolumn{2}{l}{$y \gets \arg\min_{\sigma \in \mathcal{S}'} \| \sigma - (x+e)\|  $}\\
$e \gets e + x - y $& &  $\mathcal{S}\gets  \mathcal{S}'$\\
\textbf{goto} loop\\
\bottomrule 
\end{tabularx}
\caption{The \textsl{Delayed Problem}, wherein the opponent chooses the input $x$ from the (convex hull of the) \emph{delayed version} of the set $\mathcal S$.\label{problem2}}
\end{figure}

In this paper, we show conditions on the collection of feasible sets $\SS$ under which the accumulated error $e_n$ is bounded for all $n$, and hence the average error $e_n/n$ converges to zero (cf.~\eqref{eqn:average_error}). We also propose computational methods to find tight bounds. The analysis is based on deriving \emph{bounded invariant sets} for the non-linear error dynamics
\begin{equation}
e_{n+1} = e_n + x_n - \arg \min_{y \in \S_n} \|e_n + x_n - y \|.
\end{equation}

\subsection{Related Work}
The greedy algorithm for total error minimization is also known as the \emph{error diffusion} algorithm and is well-known in the context of image processing and digital printing. Various special cases have been extensively explored in the literature.  In the classical version  (also known as \emph{Floyd-Steinberg dithering}),  the variables are one-dimensional \cite{floyd75, gray, Anastassiou, adams}. 
The extension to the general $d$-dimensional case was considered over the recent years in several papers.
In \cite{nowicki2004}, the problem of a single feasible set $\S$ is considered in the special case where $\S$ are the \emph{corner points} of a polytope $ \conv{\S}$, and an algorithm to construct a minimal invariant set for the error dynamics  is proposed. However, the boundedness of this set is not guaranteed in general.
\cite{adler2005} shows how to construct bounded invariant sets for that problem, and extends the results to a finite collection of polytopes. 
\cite{tresser2007} extends the results of \cite{adler2005} to the case where the polytope may change from step to step, and argues that there exists a bounded invariant set for the resulting time-varying dynamical system. Moreover, the conditions apply to an infinite collection of polytopes, provided that the set of face-normals of this collection is finite.
Other papers in this line of research consider specific applications \cite{boundsHalftoning, mathHalftoning, errDiffOpt} and other special cases \cite{acute}. The optimality of the error diffusion algorithm was recently analyzed in \cite{errDiffOpt}. We note that in this line of work, the term ``convex dynamics'' is  used because the error dynamics are an example of dynamics of piecewise isometries where the pieces are convex \cite{nowicki2004}.

\subsection{Contributions and Applications}
We analyze a generalized error-diffusion algorithm  with the following main contributions:
\begin{itemize}
  \setlength{\itemsep}{0pt}%
  \setlength{\parskip}{0pt}
  \item We formulate two problems involving three parties: the \textsl{undelayed} and the \textsl{delayed} one. In the latter, the input $x$ is chosen based on an outdated version of the feasible set $\S$. 
    We show that: (i) the undelayed problem is naturally associated with the $G$-dynamics analyzed in \refsec{undelayed} 
    and related to the ``dynamics in the vector space of errors'' in the terminology of \cite{nowicki2004}; and that (ii) the delayed problem is naturally associated with the $F$-dynamics analyzed in \refsec{delayed}, and related to the ``dynamics in affine space'' in the terminology of \cite{nowicki2004}.
\item We consider general non-convex sets as feasible sets $\S$, instead of restricting to the corner points of a polytope.
  The latter restriction is natural in a color printing application, where mixed colors are emulated by printing patterns of a few basic colors that are the corner points of the color space polytope (e.g., CMYK). Our focus includes other applications, such as control of dynamical systems, for which restricting the feasible set to corner points of some polytope would be undesirable.
\item We propose 
a computational method for constructing the minimal invariant set for the accumulated error dynamics in the case of finite collection of feasible sets. We also show some important special cases in which  the minimal invariant set can be computed explicitly even when the collection is uncountably infinite. As a result, we obtain tight bounds on the accumulated error.
\end{itemize}

One of the main motivations for considering the undelayed and delayed problems illustrated comes from a control application, whereby a decision maker can be viewed as a controller that is supposed to track a sequence of ``setpoints'' $\{x_n\}$ on average. The opponent can be an optimization algorithm that computes these setpoints based on a convex approximation of the feasible set, whereas the oracle can represent the uncertain elements of the problem (i.e., Nature). In \cite{errDiffCDC}, we presented preliminary results in this control context that focus on the one-dimensional case. We next briefly describe the following two examples considered there. 

\paragraph{Example 1: Control of a photovoltaic (PV) panel that is connected to the electrical grid.}
The decision maker controls the output power of the panel, and is supposed to follow a sequence of inputs provided by the grid operator, which is the opponent. The weather plays the role of the oracle: suddenly, a cloud could block the  influx of direct sunlight on the panel, by which the maximum power output of the PV panel decreases significantly.
The delayed problem variant models the case where the opponent chose its input at time $t$, while the decision maker chooses her action at time $t' > t$, at which the set of available actions could have changed (e.g., due to moving clouds). \qed

\paragraph{Example 2: Control of thermostatically controlled loads (TCLs).}
Another relevant example is control of thermostatically controlled loads (TCLs), such as air conditioning systems and water heaters. In particular, a TCL might be asked to track a continuous (time-varying) power setpoint in real time; however, it can only implement a finite number of such setpoints in practice (e.g., ON-OFF devices) at every time instant. In particular, at each time step $n$, the set of available setpoints is given by a \emph{discrete} set $\S_n$, whereas the target setpoint (e.g., computed by an optimization algorithm) lies in its convex relaxation, namely in the polytope $\conv \S_n$. 
The set $\S_n$ might change over time due to internal constraints, such as the constraints on the frequency of switching and comfort constraints related to the operation of the TCL. Moreover, there can be a communication delay, similarly to Example 1. \qed

We note that these power-related applications are of particular interest because the tracking of an average point can be interpreted as the tracking of the \emph{energy content} of the input signal (see, e.g., \cite{commelec1,opfPursuit}).

Other applications of interest were extensively reviewed in, e.g., \cite{adler2005, errDiffOpt}, and include digital printing (and, in particular, halftoning), scheduling, and assignment problems.

\subsection{Main Results}
Our main results rely on the following assumption on the collection of sets $\SS$.

\begin{assumption} \label{asm:S}
The collection \SS is such that $\conv{\SS} := \{\conv{\S}, \S \in \SS \}$ is a collection of polytopes such that:
\begin{enumerate}[(i)]
    \item The sizes of the polytopes are uniformly bounded;
    \item The set $\mathcal{N}$ of outgoing normals to the faces of the polytopes is finite; and
    \item For all  $\S \in \SS$, every bounded Voronoi cell in $\S$ is uniformly bounded;
\end{enumerate}

\end{assumption}

Part (i) of Assumption \ref{asm:S} is naturally required to obtain uniform error boundedness; and part (ii) essentially means that there is a finite number of ``shapes'' of the polytopes and is a relaxation of the assumption of finiteness of the collection $\SS$.
On the other hand, part (iii) is more technical and is illustrated in \reffig{growingvoronoi}. In particular, bounded Voronoi cells are associated with interior points of a set $\S$. These bounded Voronoi cells have to be uniformly bounded across all $\S \in \SS$. Note that this requirement does not say anything regarding the unbounded Voronoi cells (namely those associated with the boundary points of $\S$). We will give a more formal statement of part (iii) in Section \ref{sec:undelayed_bounded} (see \eqref{asm:partIII}). We note that parts (i) and (ii) of Assumption \ref{asm:S} were imposed in \cite{adler2005,tresser2007} which considered sets $\S$ that are corner points of a polytope. In that case, part (iii) is trivially satisfied, since all the Voronoi cells are associated with boundary points of $\conv (\S)$ and hence unbounded.

The following theorems summarize the main results proved in the ensuing sections.

\begin{figure}
  \vspace{-1em}
  \raisebox{.45em}{\includegraphics[width=.28\columnwidth]{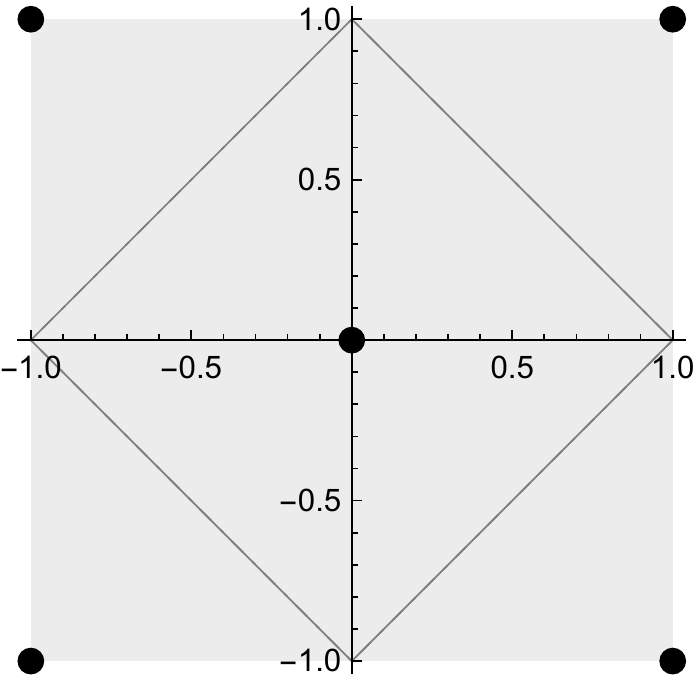}}
    \hfill
  \raisebox{.15em}{\includegraphics[width=.28\columnwidth]{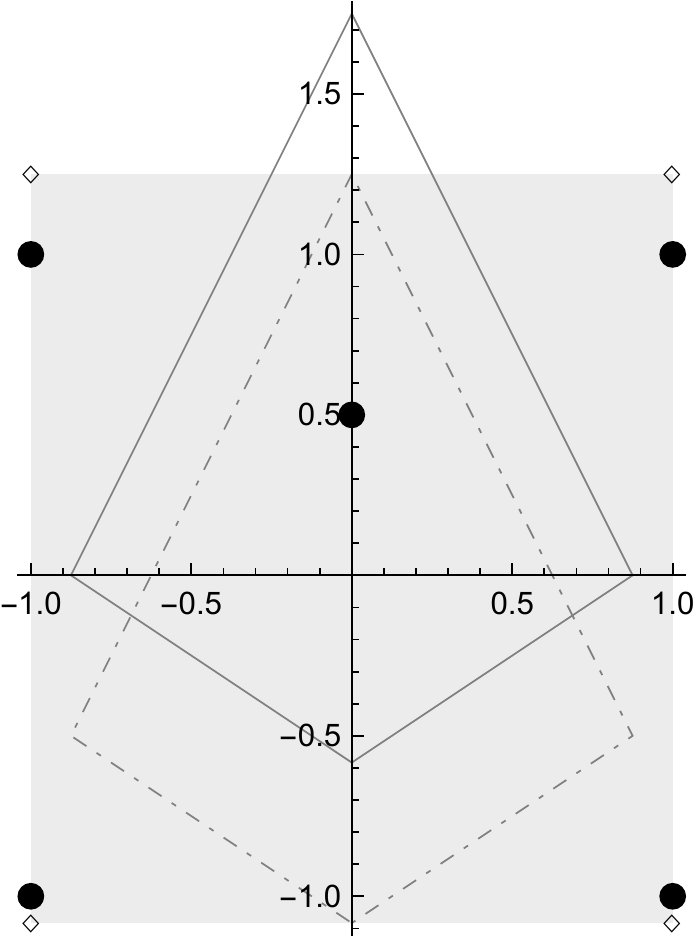}}
    \hfill
        \includegraphics[width=.28\columnwidth]{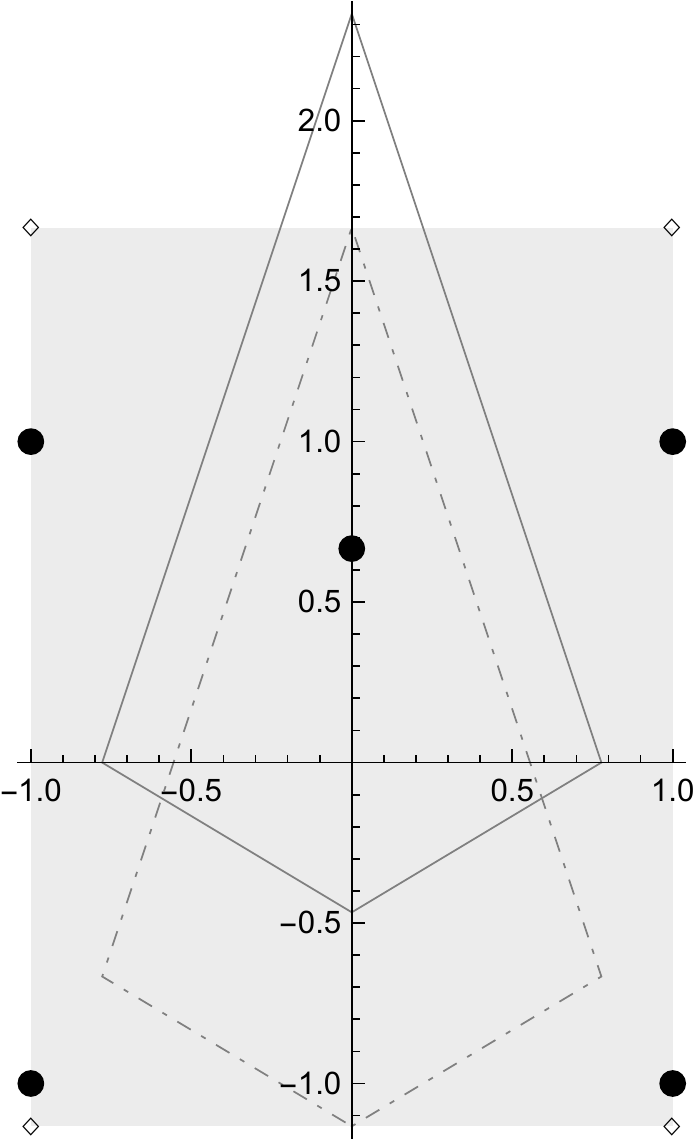}
    \caption{As the interior point, positioned at $(0,0)$ in the left subfigure, moves closer towards the boundary of the convex hull of all points, the corresponding bounded Voronoi region (shown as the solid polygon) grows. In the limit, when the distance between the interior point and the boundary becomes infinitesimally small, the growth is unbounded.\label{fig:growingvoronoi} \\
    In each subfigure, the gray rectangle shows the minimal invariant error set that corresponds to the set $\S$ in that subfigure (the constellation of the black dots), which is guaranteed (by \refprop{bounded_vor_regions}) to cover a copy of the bounded Voronoi region that is translated to the origin (shown as the dashed polygon).}
\end{figure}

\begin{theorem} \label{thm:undelayed}
Consider the undelayed problem illustrated in Figure \ref{problem1}. 
Under Assumption \ref{asm:S}, there exists a \emph{bounded} set $Q$, such that $e_n \in Q$ for all $n \geq 0$. Moreover, a \emph{minimal} such set $Q$ can be efficiently characterized via a fixed point equation.
\end{theorem}


\begin{theorem} \label{thm:delayed}
Consider the delayed problem illustrated in Figure \ref{problem2}.  
Assume that $\bigcap_{\S \in \SS} \S \neq \emptyset$. 
Under Assumption \ref{asm:S}, there exists a \emph{bounded} set $D$, such that $x_n + e_n \in D$ for all $n \geq 0$. Moreover, a \emph{minimal} such set $D$ can be efficiently characterized via a fixed point equation.
\end{theorem}

Note that these theorems obviously imply that $\|e_n\| \leq C$ for some $C < \infty$ for all $n$, and hence property \eqref{eqn:average_error} is satisfied.

\section{Preliminaries}
Throughout the paper, $\|\cdot\|$ denotes the $\ell_2$ norm. We use $\natnum$ and  $\posint$ to refer to the natural numbers including and excluding zero, respectively. For arbitrary $n\in \natnum_{>0}$, we write $[n]$ for the set $\{1,\ldots,n\}$. We use $\boldsymbol{0}$ to denote the zero vector $(0,\ldots,0)$, where the dimension should be clear from its context.

Let $d\in \natnum_{>0}$.
For arbitrary sets $\mcal{U},\mcal{V} \in \reals^d$,  $\mcal{U}+\mcal{V}$ represents the Minkowski sum of \mcal{U} and \mcal{V}, which is defined as
$\mcal{U}+\mcal{V} := \{u+v \,|\,  u \in \mcal{U},  v\in \mcal{V} \}$. Likewise, $\mcal{U}-\mcal{V}$ represents the Minkowski difference, defined as $\mcal{U}-\mcal{V} := \{u -v) \,|\,  u \in \mcal{U},  v\in \mcal{V} \}$. We let $\conv \mcal{U}$ denote the convex hull of the set $\mcal{U}$, and by $\partial \mcal{U}$ we denote the \emph{boundary} of 
\mcal{U}.
For any compact set \mcal{V}, we define the
\term{diameter of \mcal{V}} as
$
\diam \mcal{V} := \max \{ \| v-w \| : v,w \in \mcal{V} \}.
$

  Let $\mcal{S} \subset \reals^d$ be an arbitrary non-empty closed set.
  Any mapping $\vor[\mcal{S}]{x}$ that satisfies
\[
  \vor[\mcal{S}]{x} =  c, \quad\text{where}\quad
c \in \arg \min_{\rho \in\mcal{S}} \|\rho-x \|.
\]
is called  a \term{closest-point} (or \emph{projection}) operator onto \mcal{S}.
The \emph{Voronoi cell} associated with the set $\S$ and  $c \in \S$ is 
\[
V_{\S}(c) := \{x \in \reals^d: \, \|x - c\| \leq \|x - c'\|, \forall c' \in \S\}.
\]
For any  $\mcal{V} \subseteq \reals^d$, we denote the intersection of $\mcal{V}$ with $V_{\S}(c)$: 
\[
\mcal{V}_{\{c\}} := \mcal{V} \cap V_{\S}(c) 
\]
whenever the set $\S$ is clear from the context.

%
%
%





Throughout the paper, a ``set'' (or ``subset'') denotes a ``closed set'' (or ``closed subset'') unless specified otherwise explicitly.

\begin{proposition}[The Minkowski sum is distributive over unions] \label{prop:mink_union}
Let $G$ be a group and let $A$, $B$ and $C$ be arbitrary subsets of $G$. Then, 
$
A + (B \cup C) = (A + B) \cup (A + C).
$
\end{proposition}

\section{Analysis: Undelayed Problem} \label{sec:undelayed}

Recall from Figure \ref{problem1} that at each time step $n$, the input $x_n$ lies in $\conv \S_{n}$, and the decision $y_n$ is computed according to the greedy algorithm, namely 
\begin{equation} \label{eqn:err_diff}
y_n = \vor[\S_n]{e_n + x_n}.
\end{equation}
Therefore,  the dynamics for the accumulated error variable $e_n$ \eqref{eqn:e_n} is given by
\begin{equation} \label{eqn:e_n_perfect}
  e_{n+1} = e_n + x_n - \vor[\S_n]{e_n + x_n}.
\end{equation}
Similarly to \cite{adler2005, tresser2007}, 
for any non-empty 
set $\S \subset \reals^d$ and any $x \in \conv{\S}$ we define the map
\begin{align*}
G_{\S,x }:  \reals^d &\rightarrow \reals^d, \qquad  
e  \mapsto e + x - \vor[\S]{e + x}.
\end{align*}
The dynamics for the accumulated error \eqref{eqn:e_n_perfect} can be then expressed as
  \begin{equation} \label{eqn:err_diff_dyn}
  e_{n+1} = G_{\S_n,x_n }(e_n).    
  \end{equation}
For any set $A \subseteq \reals^d$, we also use $G_{\S,x }$ as a set operator on $A$, by which mean
\[
  G_{\S,x}(A) := \bigcup_{a\in A} G_{\S,x}(a).
\]

\begin{definition}[Invariance] \label{def:invariance}
We say that a set $Q \subseteq \reals^d$ is $G$-invariant with respect to a 
set $\S \subset \reals^d$
if
\[
\forall \, x \in \P, \quad G_{\S,x}(Q) \subseteq Q.
\]
We say that $Q$ is $G$-invariant with respect to a collection $\SS$ if it is $G$-invariant with respect to every $\S \in \SS$. 
\end{definition}

\begin{remark}
Our definition of invariance is a special case of
\emph{robust positively invariant sets}  in robust set-invariance theory \cite{bla99}.
Indeed, for any dynamical system $s_{n+1} = f(s_n, w_n)$, a robust positively invariant set is any set $\Omega$ that satisfies $f(\Omega, w) \subseteq \Omega$ for all $w \in \mathbb{W}$, where $\mathbb{W}$ is a set of all possible disturbances. In our case, we can view the pair $(\S_n, x_n)$ as a disturbance $w_n$ that lies in the set
\[
\mathbb{W} := \{(\S, x): \, \S \in \SS, x \in \conv \S \}.
\]
\end{remark}

The following observation makes the connection between invariant sets and boundedness of the accumulated error.
\begin{observation} \label{prop:inv_bound}
Let $\SS$ be a collection of subsets of $\reals^d$. Consider the dynamics \eqref{eqn:err_diff_dyn}, where $\S_n \in \SS$ for all $n$. Let $Q$ be a $G$-invariant set with respect to the collection $\SS$. Then, if $e_0 \in Q$, it holds that $e_n \in Q$ for all $n \geq 1$.
\end{observation}
In particular, it follows from Observation \ref{prop:inv_bound} that if $Q$ is a bounded invariant set that contains the origin, and $e_0 \in Q$, then the accumulated error is bounded for all $n$ by $\max_{v \in Q} \|v\|$.

Our next goal is to find \emph{minimal} $G$-invariant sets in order to obtain \emph{tight} bounds for the accumulated error. 
\begin{definition}[Minimality]
  Let $Q \subseteq \reals^d$ be a set that is  G-invariant with respect to a set $\S$ (respectively, a collection $\SS$), and let $Q$ contain  $A\subseteq \reals^d$.
  We say that $Q$ 
  is \emph{minimal} if it is contained in any $G$-invariant set $Q'$ (with respect to a set $\S$, respectively, a collection $\SS$) that contains $A$.
\end{definition}

In particular, we provide conditions for the existence of \emph{bounded} minimal $G$-invariant sets, and a method to compute these sets. Our analysis follows that of  \cite{nowicki2004}; however, different from \cite{nowicki2004}, we cover the case of \emph{general} non-convex sets $\S$ rather then just corner points of a polytope.

\subsection{Characterization of $G$-Invariant Sets via a Fixed-Point Equation}
We next provide a convenient characterization of $G$-invariant sets using the following set operators. 
\begin{definition}[Set operators induced by a collection of sets] \label{def:g}
Fix $d \in \natnum$.
For any 
set 
$\S \subseteq \reals^d$, respectively, a collection $\SS$ of subsets of $\reals^d$, define:
\[
  \mathfrak{g}_\S(Q):=\bigcup_{c \in \S } \vsect{(\P+Q)} -c, \quad\text{and}\quad \mathfrak{g}_\SS(Q):=\bigcup_{\S \in \SS} \mathfrak{g}_\S(Q).
\]
\end{definition}

\begin{proposition}[$G$-invariance as a fixed-point equation] \label{prop:inv_mult}
\begin{enumerate}[(i)]
\item[]
\item For any $Q \subseteq \reals^d$, 
it holds that $Q \subseteq \mathfrak{g}_\SS(Q)$.
\item $Q \subseteq \reals^d$ is $G$-invariant with respect to $\SS$ if and only if 
\[
\mathfrak{g}_\SS(Q) \subseteq Q .
\]
\end{enumerate}
\end{proposition}

To prove Proposition \ref{prop:inv_mult}, we need the following auxiliary results.

\begin{lemma} \label{lem:A_in_g}
For any $A \subseteq \reals^d$, we have that
\[
A \subseteq \mathfrak{g}_\S(A).
\]
\label{lem:Asubg}
\end{lemma}

\begin{proof}
Let $v \in A$. There exists $c \in \S$ such that $c = \vor{c + v}$. 
Obviously, $c \in \conv \S$. 
Therefore, $v = (c + v) - c = x - \vor{x}$ for $x = c + v \in (\P + A)_{\{c\}}$ and $\vor{x} = c$, implying that $v \in \mathfrak{g}_\S(A)$.
\end{proof}

\begin{proposition} \label{prop:inv}
$A \subseteq \reals^d$ is $G$-invariant with respect to $\S$ if and only if 
\[
A = \mathfrak{g}_\S(A)
\]
\end{proposition}

\begin{proof}
Note that by Lemma \ref{lem:A_in_g}, we only need to consider the condition $A \supseteq \mathfrak{g}_\S(A)$ for the purpose of the proof.

$(\Rightarrow)$ Assume that $A \supseteq \mathfrak{g}_\S(A)$. Let $x \in \P + A$. Also, by the definition of Voronoi cells, $x \in V_\S(\vor{x})$. Thus,
$
x \in (\P + A)_{\{\vor{x}\}}
$
and
\[
x - \vor{x} \in (\P + A)_{\{\vor{x}\}} - \vor{x} \subseteq \mathfrak{g}_\S(A) \subseteq A,
\]
where the last set inclusion follows by the hypothesis. 
Hence, by Definition \ref{def:invariance}, $A$ is an invariant set.

$(\Leftarrow)$ Assume that $A$ is an invariant set. Also, assume by the way of contradiction that $A \not\supseteq \mathfrak{g}_\S(A)$. Namely, there exists $v \in \mathfrak{g}_\S(A)$ such that $v \notin A$. But every $v \in \mathfrak{g}_\S(A)$ can be written as $v = x - \vor{x}$ for some $x \in \P + A$. In other words, there exists $x \in  \P + A$ such that $x - \vor{x} \notin A$, a contradiction to Definition \ref{def:invariance}.
\end{proof}

\begin{proof}[Proof of Proposition \ref{prop:inv_mult}]
By Lemma \ref{lem:A_in_g}, $A \subseteq \mathfrak{g}_{\S} (A)$ for every $\S \in \SS$. 
Hence, $A \subseteq \mathfrak{g}_\SS (A)$ and part (i) of the proposition follows.

We next prove the two directions of part (ii).

$(\Rightarrow)$ Assume $A \supseteq \mathfrak{g}_\SS (A)$. Then, also $A \supseteq \mathfrak{g}_{\S} (A)$ for all $\S \in \SS$. 
Thus, by Proposition \ref{prop:inv}, $A$ is invariant for all $\S \in \SS$, 
hence invariant with respect to $\SS$.

$(\Leftarrow)$ Assume $A$ is invariant for all $\S \in \SS$. 
Then $A \supseteq \mathfrak{g}_{\S} (A)$ all $\S \in \SS$, 
therefore also $A \supseteq \bigcup_{\S \in \SS } \mathfrak{g}_{\S} (A) = \mathfrak{g}_\SS (A)$.
\end{proof}

\subsection{Minimal $G$-Invariant Set}

The next  result shows how to find minimal $G$-invariant sets.

\begin{theorem}[Minimal Invariant Set] \label{theo:min_G}
Let $A \subseteq \reals^d$. The iterates
\[
\mathfrak{g}^n_\SS (A) := \mathfrak{g}_\SS (\mathfrak{g}^{n-1}_\SS (A)) \quad \text{for } n\in \natnum, n \geq 1, \quad \mathfrak{g}^0_\SS (A) := A,
\]
are monotonic, in the sense that $\mathfrak{g}^n_\SS (A) \subseteq \mathfrak{g}^{n'}_\SS (A)$ for all $n \leq n'$, and the limit set 
\[
\mathfrak{g}^\infty_\SS (A) := \lim_{n \rightarrow \infty} \mathfrak{g}^n_\SS (A) = \bigcup_{n\geq0} \mathfrak{g}^n_\SS(A) 
\]
is the \emph{minimal $G$-invariant set} w.r.t.~$\SS$  containing $A$.
\end{theorem}

We first need the following auxiliary results.

\begin{lemma}[Properties of $\mathfrak{g}_\SS$] \label{lem:g_prop_mult}
\begin{enumerate}
    \item[] 
    \item[(i)] (Monotonicity) If $A \subseteq B$ then $\mathfrak{g}_\SS (A) \subseteq \mathfrak{g}_\SS (B)$.
    \item[(ii)] (Additivity) $\mathfrak{g}_\SS (A \cup B) = \mathfrak{g}_\SS (A) \cup \mathfrak{g}_\SS (B)$.
\end{enumerate}
    
\end{lemma}

\begin{proof}
We start by proving these properties for a single set $\S$, namely for $\mathfrak{g}_{\S}$. Property $(i)$ is straightforward by Definition \ref{def:g}. To prove (ii), first note that by the monotonicity property (i), we have that $\mathfrak{g}_\S(A) \subseteq \mathfrak{g}_\S(A \cup B)$ and $\mathfrak{g}_\S(B) \subseteq \mathfrak{g}_\S(A \cup B)$, hence $\mathfrak{g}_\S(A) \cup \mathfrak{g}_\S(B) \subseteq \mathfrak{g}_\S(A \cup B)$. For the other direction, let $v \in \mathfrak{g}_\S(A \cup B)$. Then, $v = x - \vor{x}$ for some $x \in \P + A \cup B = (\P + A) \cup (\P + B)$, where the last equality follows by \refprop{mink_union}. Therefore, $v \in \mathfrak{g}_\S(A)$ or $v \in \mathfrak{g}_\S(B)$. This implies that $\mathfrak{g}_\S(A) \cup \mathfrak{g}_\S(B) \supseteq \mathfrak{g}_\S(A \cup B)$ and completes the proof of property (ii).

The result then follows by using the definition of $\mathfrak{g}_\SS$ as the union of $\mathfrak{g}_{\S}$ over $\S \in \SS$.
\end{proof}

\begin{proposition}[Monotonicity of the iterates] \label{prop:mono_mult}
Let $A \subseteq \reals^d$. Then the iterates $\mathfrak{g}^n_\SS (A)$ are monotonic, in the sense that $\mathfrak{g}^n_\SS (A) \subseteq \mathfrak{g}^{n'}_\SS (A)$ for all $n \leq n'$. Thus, there exists
\[
\mathfrak{g}^\infty_\SS (A) := \lim_{n \rightarrow \infty} \mathfrak{g}^n_\SS (A) = \bigcup_{n\geq0} \mathfrak{g}^n_\SS (A). 
\]
\end{proposition}

\begin{proof}
This proposition is a direct consequence of Lemma \ref{lem:g_prop_mult} (i) and \refprop{inv_mult} (i).  
\end{proof}

\begin{proof}[Proof of Theorem \ref{theo:min_G}]
We first prove the invariance of $\mathfrak{g}^\infty_\SS(A)$. We have that
\begin{eqnarray*}
\mathfrak{g}_\SS( \mathfrak{g}^\infty_\SS(A)) &=& \mathfrak{g}_\SS\Bigg( \bigcup_{n \geq 0}  \mathfrak{g}^n_\SS(A) \Bigg) 
= \bigcup_{n \geq 0} \mathfrak{g}_\SS\left(\mathfrak{g}^n_\SS(A) \right) \\
&=& \bigcup_{n \geq 1} \mathfrak{g}^n_\SS(A) 
= A \cup \bigcup_{n \geq 1} \mathfrak{g}^n_\SS(A) = \mathfrak{g}^\infty_\SS(A), 
\end{eqnarray*}
where the first equality follows by Proposition \ref{prop:mono_mult}, the second equality follows by Lemma \ref{lem:g_prop_mult} (ii), and the last equality follows by Proposition \ref{prop:inv_mult} (i) as $A \subseteq \mathfrak{g}_\SS(A)$. Thus, by Proposition \ref{prop:inv_mult},  $\mathfrak{g}^\infty_\SS(A)$ is invariant. It also contains $A$ by its definition.

To prove minimality, let $Q$ be any invariant set containing $A$. Then
$
Q \supseteq \mathfrak{g}_\SS( Q ) \supseteq \mathfrak{g}_\SS( A ),
$
where the first inclusion follows by Proposition \ref{prop:inv_mult} and the second one follows by the monotonicity property (Lemma \ref{lem:g_prop_mult} (i)). Applying these rules recursively, we obtain that
$
Q \supseteq \mathfrak{g}_\SS^n( A ), \, n \geq 0.
$
This implies that
$
Q \supseteq \cup_{n=0}^N \mathfrak{g}^n_\SS(A), \, N \geq 0
$
and hence
$
Q \supseteq \mathfrak{g}^\infty_\SS(A).
$
Therefore $\mathfrak{g}^\infty_\SS(A)$ is contained in any invariant set that contains $A$. This completes the proof of the Theorem.
\end{proof}


\subsection{Boundedness} 
\label{sec:undelayed_bounded}


The following result provides conditions on the collection $\SS$ that guarantee the boundedness of the set defined by Theorem \ref{theo:min_G}. In particular, these conditions are formulated in Assumption \ref{asm:S}. Part (iii) of this assumption can be now formally stated.
For any $\S \in \SS$, consider a partition of $\S$ imposed by $\conv S$ into two disjoint sets: 
    
    \noindent $\cor (\S)$: the ``corner points'' of $\S$, namely $\cor (\S) \subseteq \S$ such that $\cor (\S) \subseteq \partial \, \conv{S}$.
    
    \noindent $\inn (\S)$: the ``inner points'' of $\S$, namely $\inn (\S) \subseteq \S$ such that $\inn (\S) \not\subseteq \partial \, \conv{S}$.
    
Note that in this definition, we include in the corner points also any points in $\S$ that lie on the edges of $\conv{\S}$. It is easy to see that the Voronoi cells associated with $\inn (\S)$ are bounded, while the Voronoi cells associated with $\cor (\S)$ are unbounded. Moreover, $\conv{\S} = \conv{\cor (\S)}$. Now, part (iii) of Assumption \ref{asm:S} can be stated as
\begin{equation} \label{asm:partIII}
\sup_{\S \in \SS} \sup_{c \in \inn (\S)} \diam (V_{\S}(c)) < \infty. 
\end{equation}
\begin{theorem}[Bounded Invariant Set] \label{theo:bound_G}
Under Assumption \ref{asm:S}, the minimal $G$-invariant set with respect to $\SS$ that contains the origin is \emph{bounded}. 
\end{theorem}


To prove this theorem, we need the following basic result about the Voronoi cells of $\S$ and $\cor(\S)$.

\begin{lemma} \label{lem:vor_Sc}
We have that $V_\S(c) \subseteq V_{\cor(\S)}(c)$ for all $c \in \cor(\S)$.
\end{lemma}

\begin{proof}
Let $c \in \cor(\S)$ and $x \in V_\S(c)$. By the definition of the Voronoi cell
\[
\|x - c\| \leq \|x - c'\|, \quad \forall c' \in \S.
\]
However, since $\cor(\S) \subseteq \S$, it is also true that
\[
\|x - c\| \leq \|x - c'\|, \quad \forall c' \in \cor(\S),
\]
implying that $x \in V_{\cor(\S)}(c)$.
\end{proof}
\begin{proof}[Proof of Theorem \ref{theo:bound_G}]

In \cite{adler2005, tresser2007}, it was shown that for the collection $\cor(\SS) := \{\cor(\S), \S \in \SS\}$, under the hypothesis of Theorem \ref{theo:bound_G},  for any $r \geq 0$ large enough, one can construct a \emph{bounded convex} invariant set $Q_r$ which contains a ball with radius $r$ centered in the origin (Theorem 2.1 in \cite{tresser2007}). That is
\begin{equation} \label{eqn:inv_Sc}
\forall \S \in \SS, \quad \mathfrak{g}_{\cor(\S)}(Q_r) = Q_r.
\end{equation}
We next extend this proof to the case where $\inn(\S) \neq \emptyset$.

By the definition of $\mathfrak{g}$, we have that 
\begin{align*}
\mathfrak{g}_{\S}(Q_r) &=  \bigcup_{c \in \S } (\P+Q_r) \cap V_{\S}(c) -c \\
&= \left [ \bigcup_{c \in \cor(\S) } (\P+Q_r) \cap V_{\S}(c) -c \right ] \\
&\phantom{=} \bigcup \left [ \bigcup_{c \in \inn(\S) } (\P+Q_r) \cap V_{\S}(c) -c \right ] \\
&\subseteq \left [ \bigcup_{c \in \cor(\S) } (\P+Q_r) \cap V_{\cor(\S)}(c) -c \right ] \\
&\phantom{=}\bigcup \left [ \bigcup_{c \in \inn(\S) } (\P+Q_r) \cap V_{\S}(c) -c \right ] \\
&= \mathfrak{g}_{\cor(\S)}(Q_r) \bigcup \left [ \bigcup_{c \in \inn(\S) } (\P+Q_r) \cap V_{\S}(c) -c \right ], 
\end{align*}
where the set inclusion follows by \reflem{vor_Sc}. 
Now observe that there exists $r_0 < \infty$ such that for all $r \geq r_0$,
\[
(\P+Q_r) \cap V_{\S}(c)  = V_{\S}(c), \quad c \in \inn(\S)
\]
as the Voronoi cells for $c \in \inn(\S)$ are bounded. Thus, from the uniform boundedness of bounded Voronoi cells \eqref{asm:partIII}, for $r \geq r_0$, 
\[
\bigcup_{c \in \inn(\S) } (\P+Q_r) \cap V_{\S}(c) -c
\]
is a bounded set. Moreover, since the Voronoi cells for $c \in \S_{c}$ are unbounded, 
\[
\lim_{r \rightarrow \infty} \bigcup_{c \in \cor(\S) } (\P+B(0, r)) \cap V_{\cor(\S)}(c) -c = \reals^d,
\]
where $B(0,r)$ denotes a ball of radius $r$ centered at the origin.
Hence, there exists $r_1 \geq r_0$, such that for all $r \geq r_1$
  \begin{align}
\nonumber    \bigcup_{c \in \inn(\S) } & (\P  +Q_r) \cap V_{\S}(c) -c \\
\nonumber  &\subseteq \bigcup_{c \in \cor(\S) } (\P+B(0, r)) \cap V_{\cor(\S)}(c) -c  \\
\nonumber  &\subseteq \bigcup_{c \in \cor(\S) } (\P+Q_r) \cap V_{\cor(\S)}(c) -c  \\
\label{eqn:unbounded_domin}
&=\mathfrak{g}_{\cor(\S)}(Q_r) . 
\end{align}

By the uniform boundedness of $V_\S(c)$ for $c\in \inn(\S)$ and $\S \in \SS$ \eqref{asm:partIII}, it follows that there exists $r_1^* < \infty$ such that \eqref{eqn:unbounded_domin} holds for all $r \geq r_1^*$ and all $\S \in \SS$.
Thus, for all $\S \in \SS$, 
 $\mathfrak{g}_{\S}(Q_r) \subseteq \mathfrak{g}_{\cor(\S)}(Q_r) = Q_r$, where the last equality follows by the invariance of $Q_r$ for $\cor(\SS)$ \eqref{eqn:inv_Sc}. Therefore,  for $r \geq r_1^*$,
\[
\mathfrak{g}_{\SS}(Q_r) = \bigcup_{\S \in \SS} \mathfrak{g}_{\S}(Q_r) \subseteq Q_r,
\]


\noindent implying that $Q_r$ is a $G$-invariant set with respect to $\SS$ by \refprop{inv_mult}. 
Since $Q_r$ is a bounded convex set, the result of  Theorem \ref{theo:bound_G}  follows.
\end{proof}


\subsection{Other Properties of $G$-Invariant Sets}

\begin{definition}[\cite{moszynska2006}]
  We say that a set $A\subseteq \reals^d$ is \emph{star-convex} in $x \in A$ 
if for all $\alpha \in A$, the line segment connecting $x$ and $\alpha$ is contained in $A$.
The \emph{kernel} of $A$ is the set of all points $z \in A$ such that $A$ is star-convex in $z$.
\end{definition}
\begin{theorem} \label{thm:starconvex}
Let $A \subseteq \reals^d $ be  star-convex in the origin.
Let $\S$ and $\SS$ be, respectively, a subset and a collection of subsets of $\reals^d$.
Then, $\mathfrak{g}_\S(A)$ and $\mathfrak{g}_\SS(A)$ are star-convex in the origin.
\end{theorem} 
\begin{corollary} 
  The minimal $G$-invariant set (with respect to $\S$ respectively $\SS$) that contains the origin is star-convex.
\end{corollary} 
To prove \refthm{starconvex}, we will need the following results, which are well-known hence stated without proofs.
\begin{proposition} 
Any non-empty convex set $A\subseteq \reals^d $ is star-convex in any point $x\in A$. \label{prop:conv_is_starconv}
\end{proposition} 
\begin{corollary} 
The kernel of a non-empty convex set $A\subseteq \reals^d $ is $A$ itself.
\end{corollary} 

\begin{proposition}[\cite{moszynska2006}]
Let $A\subseteq \reals^d $ and $B\subseteq \reals^d $ be star-convex sets in $x \in A \cap B$.
Then, $A \cup B$ and $A \cap B$ are star-convex in $x$.
\label{prop:starconv_union}
\label{prop:starconv_intersection}
\end{proposition}

\begin{proposition}[\cite{krantz2014}, see \cite{Li95} for a proof for the case of polygons] 
Let $A\subseteq \reals^d $ and $B\subseteq \reals^d $ be star-convex sets with kernels $k_A$ and $k_B$ respectively.
Then, it holds that $A+B$ is a star-convex set, with kernel $k_A + k_B$. 
\label{prop:starconv_minkowski}
\end{proposition}

\begin{proof}[Proof of Theorem~\ref{thm:starconvex}]
The set $\conv \S$ is a convex set and $A$ is star-convex in the origin. Hence, by \refprop{conv_is_starconv} and \refprop{starconv_minkowski} it holds that $\conv \S + A$ is star-convex in any point of $\conv \S$. Fix an arbitrary $c \in \S$. In particular, $\conv \S + A$ is star-convex in $c$. The Voronoi region $V_\S(c)$ is convex and always contains the point $c$ itself. Now, by applying \refprop{conv_is_starconv} and \refprop{starconv_intersection}, the set $(\S + A) \cap V_\S(c)$ is star convex in the point $c$. Translating this set by the vector $-c$ yields a set that is star convex in the origin. 
We picked $c$ arbitrarily, thus we may conclude that $(\S + A)_{\{c\}} - \{c\}$ is star-convex in the origin for any $c \in \S$. We now apply \refprop{starconv_union} which gives us the claim for $\mathfrak{g}_\S$, and another invocation of \refprop{starconv_union} yields the claim for $\mathfrak{g}_\SS$.
\end{proof}


The minimal $G$-invariant set has the following interesting property, 
which we proved by making use of the star-convexity of the minimal $G$-invariant set. 
In Section \ref{sec:num}, we give some illustrative examples of this property.
\begin{proposition}[Minimal Invariant Set Covers Translated Bounded Voronoi Regions]
  \label{prop:bounded_vor_regions}
Let $Q$ be the minimal $G$-invariant set with respect to $\mathcal S$ that contains the origin, and let $\mathcal S_{nc}:= \mathcal S \setminus (\partial\, \conv \mathcal S)$ be the ``non-corner'' points of $\mathcal S$. Then, 
\[
  Q \supseteq V_{\mathcal S}(y) - \{ y \} \quad \text {for all } y \in \mathcal S_{nc}.
\]
\end{proposition}
In \reffig{proofill}, we give an example illustration which might be helpful to get an intuition for the proof.
\begin{figure}
  \centering
  \includegraphics[width=0.2\textwidth]{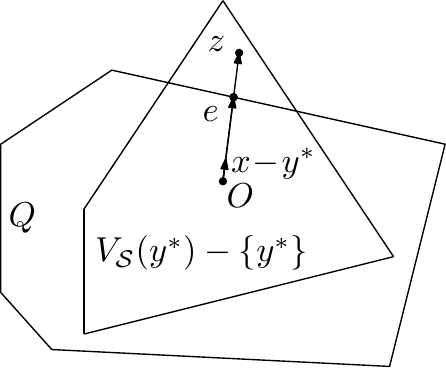}
  \caption{\label{fig:proofill}Illustration of the proof of \refprop{bounded_vor_regions}.}
\end{figure}
\begin{proof}
  Assume for the sake of contradiction that there exists $y^* \in \mathcal S_{nc}$ for which $ Q \not\supseteq V_\mathcal{S} (y^*) -\{y^* \}$. Let $W:=  V_\mathcal{S} (y^*) -\{y^* \}$.
  Choose $z \in W \setminus Q $. 
  Let $e \in \partial Q \cap W$ be the point on the line through $z$ and the origin that is 
   closest  to $z$ (in Euclidean distance).
  Note that by the fact that $Q$ is star-convex in the origin, by the convexity of $W$, and by the fact that $W$
  also contains the origin,
  such a point $e$ is well-defined.
Because $y^* \in \mathcal S_{nc}$ lies in the interior of $\conv \S$, for small enough $\varepsilon > 0$ there exists
$x \in \conv \S$ such that 
the following three properties hold: (i) $\|x-y^*\|=\varepsilon$, (ii) the difference vector $x-y^*$ points in the same direction as $z$ (and, therefore, $e$), and (iii)  $e + (x - y^*)$ lies on the line segment between $e \in W$
and $z \in W$. 

%
%
 By convexity of $V_\mathcal{S}(y^*)$ and by construction of $x$ and $e$, it holds that  $x+e-y^* \in W =  V_\mathcal{S}(y^*)- \{y^* \}$, or $x+e \in V_\mathcal{S}(y^*)$.  
However, since $e$ lies already on the boundary of $Q$, 
the vector $x+e-y^* \notin Q$ 
which contradicts the invariance of $Q$.
\end{proof}

\section{Analysis: Delayed Problem} \label{sec:delayed}
We now analyze the problem illustrated in Figure \ref{problem2}, wherein the input $x_n$ at time step $n$ is chosen from the \emph{delayed} convex set $ \conv \S_{n-1}$. We then have the following error dynamics
\[ 
e_{n+1} = e_n + x_n - \vor[\S_n]{e_n + x_n} 
\]
where $x_n \in \conv{\S_{n-1}}$ and $\S_n, \S_{n-1} \in \SS$. 
Observe that this dynamics involves a pair of sets rather then a single set, and hence 
the previous results cannot be applied directly. Fortunately, we can reformulate this dynamics in terms of the modified input $z_n = e_n + x_n$ similarly to the approach in \cite{adler2005,nowicki2004}\footnote{Although similar approach was used in \cite{adler2005,nowicki2004}, it was applied to a different problem; to the best of our knowledge, we are the first to analyze the delayed problem using this reformulation.}. For this new state variable, we have that
\begin{equation} \label{eqn:x_dyn}
z_{n+1} = z_n + x_{n+1} - \vor[\S_n]{z_n},
\end{equation}
where $x_{n+1} \in \conv{\S_{n}}$. The advantage of reformulation \eqref{eqn:x_dyn} is that it depends on the single set $\S_n$ rather than the pair $(\S_n, \S_{n-1})$. Thus, the following operator can be defined.

\begin{definition}
For any finite non-empty  set $\S \subset \reals^d$ and any $x \in \conv{\S}$ we define the map
\begin{align*}
F_{\S,x }:  \reals^d &\rightarrow \reals^d  \\
z & \mapsto z + x - \vor[\S]{z}
\end{align*}
\end{definition}

The dynamics \eqref{eqn:x_dyn} can be then expressed as
\begin{equation} \label{eqn:x_dyn2}
z_{n+1} = F_{\S_n,x_{n+1}}(z_n)
\end{equation}
and $F$-invariance with respect to $\SS$ is given by Definition \ref{def:invariance} by replacing $G$ with $F$.

The following proposition makes the connection between $F$-invariant sets and boundedness of the accumulated error.
\begin{proposition} \label{prop:inv_bound_domain}
Let $\SS$ be a collection of subsets of $\reals^d$. Consider the dynamics \eqref{eqn:x_dyn2}, where $\S_n \in \SS$ for all $n$. Let $D$ be an $F$-invariant set with respect to the collection $\SS$, and assume that $z_0 \in D$.
\begin{enumerate}
    \item[(i)] Then it holds that 
    $\|e_n\| \leq \max_{\S \in \SS} \max_{v \in D - \conv \S} \|v\|$
    for all $n \geq 1$.
    \item[(ii)] If in addition $\S \subseteq D$ for all $\S \in \SS$, we have that $\|e_n\| \leq \diam D$ for all $n \geq 1$.
\end{enumerate}
\end{proposition}

\begin{proof}
Part (i) of the proposition follows trivially by the invariance of $D$, and the fact that $e_n = z_n - x_n$ for $z_n \in D$ and $x_n \in \conv \S$ for some $\S \in \SS$.
For part (ii), observe that
\[
e_{n+1} = x_n + e_n - \vor[\S_n]{x_n + e_n} = z_n - \vor[\S_n]{z_n}
\]
where both $z_n \in D$ and $\vor[\S_n]{z_n} \in \S_n \subseteq D$. Hence, $\|e_{n+1}\| \leq \diam D$.
\end{proof}

We next formulate results that provide conditions for the existence of bounded minimal $F$-invariant sets, and a method to compute these sets. 

\subsection{Characterization of $F$-Invariant Sets via a Fixed-Point Equation}

We first define set-operators similarly to \cite{nowicki2004,tresser2007}.

\begin{definition}[Set operators] \label{def:p}
Fix $d \in \natnum$.
For any set $\S \subseteq \reals^d$, we define:
\[
\mathfrak{p}_\S(D):=\P + \bigcup_{c \in \S } \vsect{D} -c.
\]
Also, for a collection \SS, let
\[
\mathfrak{p}_\SS (D):=\bigcup_{\S \in \SS} \mathfrak{p}_\S(D).
\]
Finally, we define the iterates $\mathfrak{p}^n_\S(D)$ and  $\mathfrak{p}^n_\SS(D)$ in the same way as we did for the $G$-iteration in Theorem~\ref{theo:min_G}. 
 \\
\end{definition}

\begin{proposition} \label{prop:inv_mult_p}
For any set $D \subseteq \reals^d$, it holds that $D \subseteq \mathfrak{p}_\SS(D)$, and $D$ is $F$-invariant with respect to $\SS$ if and only if 
\[
\mathfrak{p}_\SS(D) \subseteq D. 
\]
\end{proposition}

 To prove this proposition, we need the following two results
that were proven in  \cite{nowicki2004} for the case where a set $\S$ is the collection of \emph{corner points} (vertexes) of $\conv \S$. These proofs extend exactly in the same way as these of $G$-invariance provided in \refsec{undelayed}; thus,  we present these two results without a proof.

\begin{lemma}[Extension of Lemma 3.11 in \cite{nowicki2004}] \label{lem:p_prop}
\begin{enumerate}
    \item[] 
    \item[(i)] (Monotonicity) If $A \subseteq B$ then $\mathfrak{p}_\SS(A) \subseteq \mathfrak{p}_\SS(B)$.
    \item[(ii)] (Additivity) $\mathfrak{p}_\SS(A \cup B) = \mathfrak{p}_\SS(A) \cup \mathfrak{p}_\SS(B)$.
\end{enumerate}
\end{lemma}

\begin{proposition}[Extension of Lemma 3.2 and Proposition 3.12 in \cite{nowicki2004}] \label{prop:inv_p}
Any $D \subseteq \reals^d$ satisfies $D \subseteq \mathfrak{p}_\SS(D)$. Moreover, 
$D$ is $F$-invariant with respect to $\SS$ if and only if 
$
\mathfrak{p}_\SS(D) \subseteq D.
$
\end{proposition}

\begin{proof}[Proof of Proposition \ref{prop:inv_mult_p}]
The proof follows the lines of the proofs of Proposition \ref{prop:inv_mult} and \ref{prop:inv_conv} by leveraging the results of Lemma \ref{lem:p_prop} and Proposition \ref{prop:inv_p}.
\end{proof}

\subsection{Minimality and Boundedness}

\begin{theorem} \label{thm:main_uncertain}
Let \SS be a collection of sets $\S \subset \reals^d$, and let $D$ be any given set. 
Then the iterate $\mathfrak{p}^n_\SS(D)$ is monotonically non-decreasing, and the set
\[
\mathfrak{p}^\infty_\SS(D) := \lim_{n \rightarrow \infty} \mathfrak{p}^n_\SS(D) 
\]
is the minimal $F$-invariant set that contains the set $D$. 
\end{theorem}

\begin{theorem} \label{thm:bounded_F}
Let $\SS$ be a collection of non-empty subsets of $\reals^d$. Assume that $\bigcap_{\S \in \SS} \S \neq \emptyset$. 
Under the conditions of Theorem \ref{theo:bound_G}, for any $s_0 \in \bigcap_{\S \in \SS} \S$, $\mathfrak{p}^\infty_\SS(\{s_0\})$ 
is a bounded set that contains $\cup_{\S \in \SS} \conv \S $.
\end{theorem}

We note that Theorem \ref{thm:bounded_F} allows us to invoke part (ii) of Proposition \ref{prop:inv_bound_domain} to obtain an explicit bound on the accumulated error as follows.

\begin{corollary}
Suppose that $e_0 = 0$. Under the conditions of Theorem \ref{thm:bounded_F}, for any $x_0 \in \cup_{\S \in \SS} \conv \S$, we have that
\[
\| e_n \| \leq \diam \l ( \mathfrak{p}^\infty_\SS(\{s_0\})\r), \quad n \geq 1, 
\]
for any $s_0 \in \bigcap_{\S \in \SS} \S$.
\end{corollary}

\begin{proof}
Since $z_0 = x_0 + e_0 = x_0 \in \cup_{\S \in \SS} \conv \S$, by Theorem \ref{thm:bounded_F}, $z_n \in \mathfrak{p}^\infty_\SS(\{s_0\})$ for all $n$. In particular, the latter  set contains $\cup_{\S \in \SS} \conv \S$, and the result follows by  Proposition \ref{prop:inv_bound_domain} (ii).
\end{proof}

To prove Theorem \ref{thm:main_uncertain}, we need the following extension of the results in \cite{nowicki2004}, that are provided here without a proof.
\begin{proposition}[Monotonicity of the iterates; Extension of Lemma 3.13 in \cite{nowicki2004}] \label{prop:mono_p}
Let $D \subseteq \reals^d$. Then the iterates $\mathfrak{p}^n_\SS(D)$ are monotonic, in the sense that $\mathfrak{p}^n_\SS(D) \subseteq \mathfrak{p}^{n'}_\SS(D)$ for all $n \leq n'$. Thus, there exists
\[
\mathfrak{p}^\infty_\SS(D) := \lim_{n \rightarrow \infty} \mathfrak{p}^n_\SS(D) = \bigcup_{n\geq0} \mathfrak{p}^n_\SS(D).
\] 
\end{proposition}
\begin{proposition}[Extension of Corollary 3.15 in \cite{nowicki2004}] \label{prop:main_single_p}
Let $D \subseteq \reals^d$. The set $\mathfrak{p}^\infty_\SS(D)$ is the \emph{minimal} $F$-invariant set containing the set $D$. 
\end{proposition}
To prove Theorem \ref{thm:bounded_F}, we need the following  result.
\begin{lemma} \label{lem:union_min}
Let $\SS$ be a collection of non-empty subsets of $\reals^d$. Assume that $\bigcap_{\S \in \SS} \S \neq \emptyset$. Let 
\[
D := \bigcup_{\S \in \SS} \conv S.
\]
Then, for any $s_0 \in \bigcap_{\S \in \SS} \S$, we have that
$
\mathfrak{p}_\SS(\{s_0\}) = D.
$
Moreover, if $D$ is $F$-invariant with respect to $\SS$, it is the \emph{minimal} $F$-invariant set with respect to $\SS$ that contains $s_0$. 
\end{lemma}

\begin{proof}
Let $s_0 \in \bigcap_{\S \in \SS} \S$. We use the iteration of \refthm{main_uncertain} to prove that $D$ is the minimal $F$-invariant set that contains $s_0$. For the first iteration, for every $\S \in \SS$, we have by \refdef{p} that
\begin{eqnarray*}
\mathfrak{p}_\S(\{s_0\}) &:=& \P + \bigcup_{c \in \S } \vsect{\{s_0\}} -c \\
&=& \P + s_0 - s_0 = \P, 
\end{eqnarray*}
where the second equality follows by the fact that $s_0 \in \S$. Therefore, 
$
\mathfrak{p}_\SS(\{s_0\}) := \bigcup_{\S \in \SS} \mathfrak{p}_\S(\{s_0\}) = \bigcup_{\S \in \SS} \P := D.
$
Now for the second iteration,
$
\mathfrak{p}_\SS^2(\{s_0\}) = \mathfrak{p}_\SS(D).
$
Therefore, if $D$ is invariant, we have that $\mathfrak{p}_\SS^2(\{s_0\}) = D$, implying that the iteration has converged, and by \refthm{main_uncertain}, $D$ is the minimal $F$-invariant set with respect to $\SS$ that contains $s_0$.
\end{proof}

\begin{proof}[Proof of Theorem \ref{thm:bounded_F}]
The proof of boundedness of $\mathfrak{p}^\infty_\SS(\{s_0\})$ 
leverages the results of \cite{adler2005,tresser2007}, particularly Theorem 2.1 in \cite{tresser2007}, and follows that of Theorem \ref{theo:bound_G}. The fact that these sets contain $\cup_{\S \in \SS} \conv \S $ follows by Lemma \ref{lem:union_min} and the monotonicity of the iterates (Propositions \ref{prop:mono_p}).
\end{proof}

\section{A Computational Method for Computing Invariant Sets}

We can turn the iteration of Theorem \ref{theo:min_G} ($\mathfrak{g}^{n}$) 
or \refthm{main_uncertain} ($\mathfrak{p}^{n}$) 
into a computational method (an algorithm that does not necessarily terminate) by augmenting the iteration with the stopping rule that corresponds to the invariance property (\refprop{inv_mult}): ending the iteration when the vertex-representation of 
the $(n+1)$th and $n$th iterates are equal. 
We have implemented the method for the special case of point sets in $\reals^2$.
The implementation is written in C++ with the help of the CGAL library; 
the source code is available online \cite{boumangit}. To prevent loss of precision during the iterations, and to be able to perform exact equality tests, we use exact rational arithmetic, instead of floating-point arithmetic.
%
%
Note that this choice restricts all vertices to have coordinates in $\mathbb{Q}$.

A typical problem with exact arithmetic is coefficient growth during the computation.
In the particular case of rational arithmetic, a vertex coordinate might \emph{approach} a
mixed number (a sum of an integer and a proper fraction) whose fractional part has small numerator and denominator (by ``small'' we mean just a few digits),
like $1/2$, $5\tfrac13$, $300\tfrac56$, etc., but never reach it in finite time. 
As a practical way to mitigate this problem, we apply, once in every $r \in \natnum$ iterations, the following conditional rounding function to each coordinate (of every vertex of the iterate) 
for which the maximum of the 
number of decimal digits of the numerator and denominator of that coordinate exceeds a threshold $s\in \natnum$.
For a positive $\epsilon \in \mathbb{R}$ (say, $\epsilon = 10^{-8}$) and a set $\mathcal{X}_k := \{ a/b : a,b \in \mathbb N ,\, 0\leq a \leq b,\, 1\leq b \leq k \}$ 
of proper fractions with small numerators and denominators (and including $0$ and $1$), 
parameterized by some $k\in \mathbb N$ (in practice, we pick $k \in [10^2, 10^3]$), 
we define
\begin{align*} 
\mathcal{R}:  \mathbb{Q} &\rightarrow \mathbb{Q} \\
q & \mapsto \begin{cases}
\lfloor q \rfloor + t & \text{if } |t- (q\!\mod 1) | \leq \epsilon, \\
q & \text{otherwise}, 
\end{cases}
\end{align*}
\noindent where  $t:= \argmin_{x \in \mathcal{X}_k }|x-(q\!\mod 1)|$.

The parameter $r$ has been introduced to prevent the rounding operation from hindering convergence, which can happen if the rounding operation keeps undoing a perturbation of the iterate applied by the set operator.
The parameter $s$ ensures that rounding only ``kicks in'' when it is really necessary, i.e., when sufficient coefficient-growth has occurred. Both parameters should not be chosen too small; $r=10$ and $s=20$ works well for us.

Recall that it immediately follows from the stopping rule that if the method converges, it means that it has found an invariant set. Hence, we may in principle perturb the set in an arbitrary way after each iteration in an attempt to aid convergence, provided that we ensure that any such ``ad-hoc'' perturbation cannot interfere with the perturbations performed by the set operator and with the procedure to check for convergence. The ``rounding trick'' outlined above works well in practice.

Note, however, that by perturbing the iterate 
(through $\mathcal{R}$) during the iteration, we cannot guarantee anymore that the method finds the \emph{minimal} invariant set. Nonetheless, if the method converges  and  
coordinate rounding occurs \emph{only} just before convergence, in other words, if 
the last ``rounding-free'' iterate is $\delta$-close (measured by a suitable metric for sets)   
to the invariant set found by the method, then that invariant set is 
a $\delta$-close approximation to the minimal invariant set, by the monotonicity the iterates.

An additional benefit of applying $\mathcal{R}$ 
is that the method is likely to  find an invariant set whose vertex-coordinates have small representation. 



\section{Examples in $\mathbb{R}^2$} \label{sec:num}
In a first set of examples shown in \reffig{evo1} and \reffig{singleset1}, we let $\mathbb{S} = \{ \mcal{S} \}$. The $G$-invariant error sets have been found with our computational method. 

\begin{figure}[t]
  \begin{subfigure}[b]{.24\textwidth}
  \centering
        \includegraphics[scale=.6,page=1]{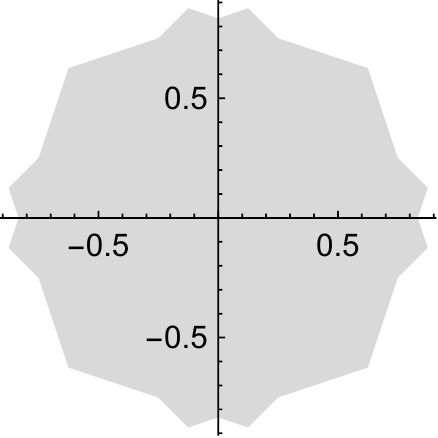}
        \caption{First iteration.}
        \label{fig:fulliter1}
    \end{subfigure}
    \hfill
  \begin{subfigure}[b]{.24\textwidth}
  \centering
        \includegraphics[scale=.6,page=2]{figs/joined}
        \caption{Second iteration.}
        \label{fig:fulliter2}
    \end{subfigure}
    \hfill
  \begin{subfigure}[b]{.24\textwidth}
  \centering
        \includegraphics[scale=.6,page=3]{figs/joined}
        \caption{Third iteration.}
        \label{fig:fulliter3}
    \end{subfigure}
    \hfill
  \begin{subfigure}[b]{.24\textwidth}
  \centering
        \includegraphics[scale=.6,page=4]{figs/joined}
        \caption{Fourth iteration.}
        \label{fig:fulliter4}
    \end{subfigure}
    \caption{Evolution of the iterate for the $G$-iteration ($\mathfrak{g}_\S$) for point set  $\S :=\{ (\pm 2,0), (0,\pm 2), (\pm \tfrac12,\pm \tfrac12) \}.$ Convergence in the fourth iteration.\label{fig:evo1}}
\end{figure}
\begin{figure}[t]
  \begin{subfigure}{.48\columnwidth}
  \centering
        \includegraphics[scale=.5]{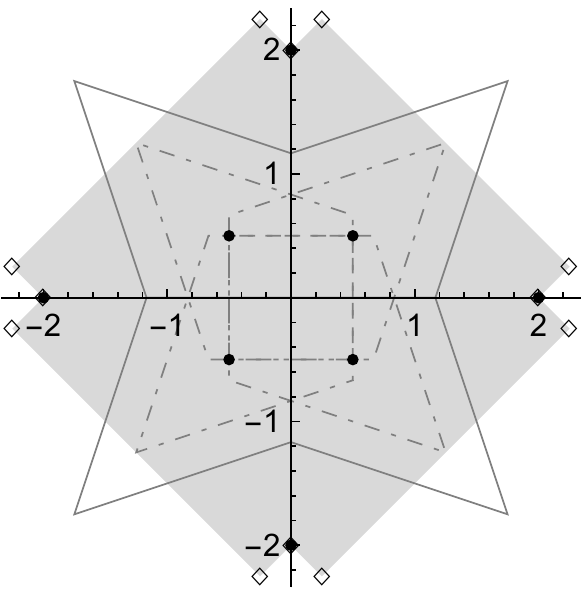}
        \caption{\raggedright $\S :=\{ (\pm 2,0), (0,\pm 2), (\pm \tfrac12,\pm \tfrac12) \}$. Minimal invariant set reached after $4$ iterations.}
        \label{fig:sset1}
    \end{subfigure}
    \hfill
 \begin{subfigure}{.48\columnwidth}
  \centering
  \includegraphics[scale=.5]{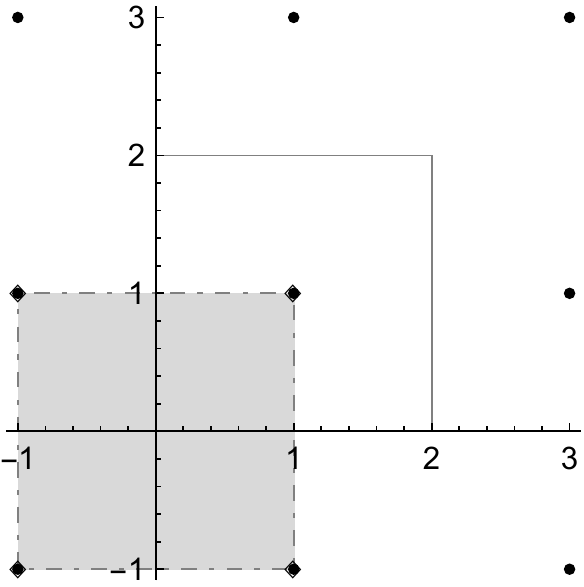} 
\caption{\raggedright $\S :=\{ - 1,1,3\}^2 $. Minimal invariant set reached after $1$ iteration.}
\label{fig:sset4}
    \end{subfigure}\\[2em]
  \begin{subfigure}{.46\columnwidth}
  \centering
  \includegraphics[scale=.6]{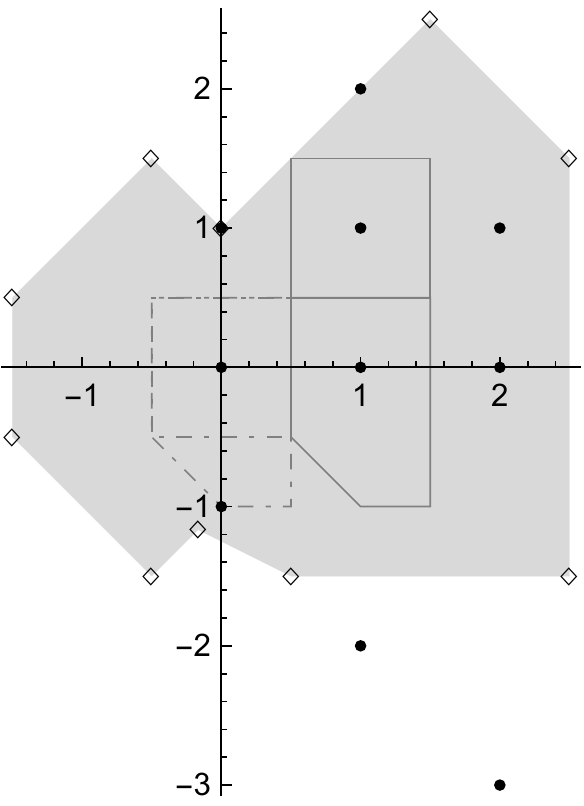}
  \caption{\raggedright $\S :=(\{ 0,1,2\} \times \{0,1\}) \cup \{(1,2), (0,-1),(1,-2),(2,-3) \} $. ``Nearly minimal'' invariant set (some coordinate rounding occurred near convergence) reached after $63$ iterations.}
  \label{fig:sset3}
    \end{subfigure}
    \hfill
 \begin{subfigure}{.48\columnwidth}
  \centering
  \includegraphics[scale=.6]{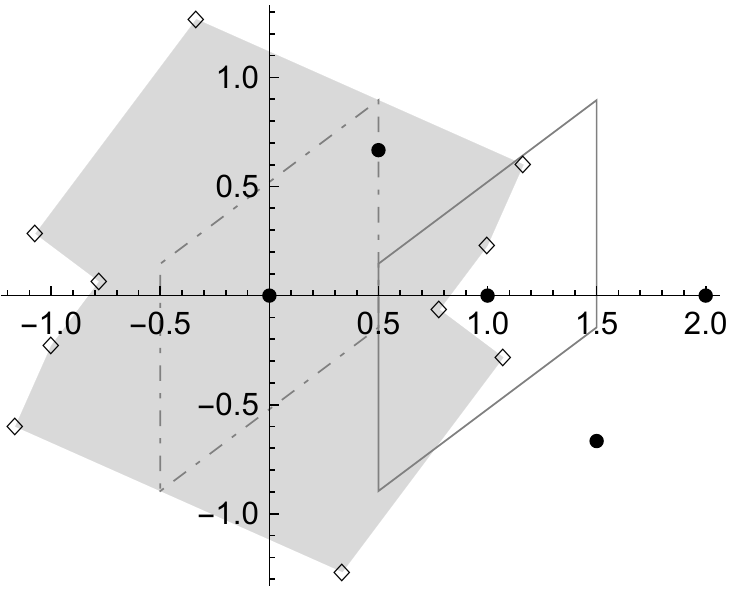}\\
  \caption{\raggedright $\S :=\{(0,0), (\tfrac12, \tfrac23),(1,0),$ $(\tfrac32,-\tfrac23),(2,0)\}$. Minimal invariant set reached after $6$ iterations.}
  \label{fig:sset2}
    \end{subfigure}
    \caption{Examples of $G$-invariant sets that correspond to single-set collections $\mathbb S :=\{\mathcal{S}\}$.
    The black dots indicate the points, the shaded areas represent the invariant sets and the diamonds represent their corner points, the solid gray lines indicate the bounded Voronoi regions, and the dashed lines show those regions translated by their corresponding Voronoi vertices. Note that \refprop{bounded_vor_regions} asserts that the invariant set covers all translated bounded Voronoi cells.\label{fig:singleset1}} 
\end{figure}

\begin{figure}
  \begin{subfigure}[b]{.30\columnwidth}
  \centering
  \includegraphics[scale=.6]{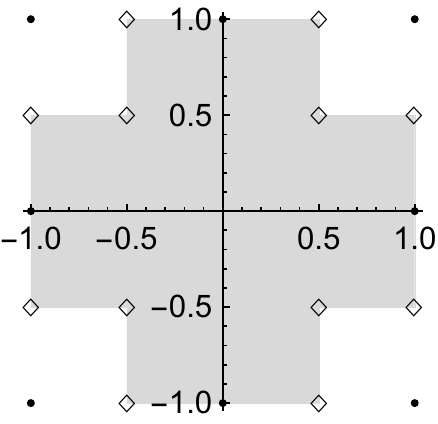}
      \caption{$\S_1$}
      \label{fig:pts1}
    \end{subfigure}
    \hfill
  \begin{subfigure}[b]{.30\columnwidth}
  \centering
        \includegraphics[scale=.6]{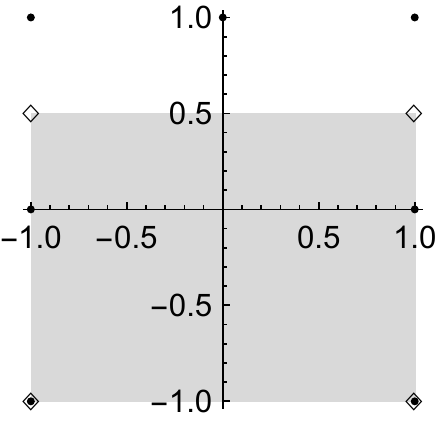}
      \caption{$\S_2$}
      \label{fig:pts2}
    \end{subfigure}
\hfill
  \begin{subfigure}[b]{.30\columnwidth}
  \centering
        \includegraphics[scale=.6]{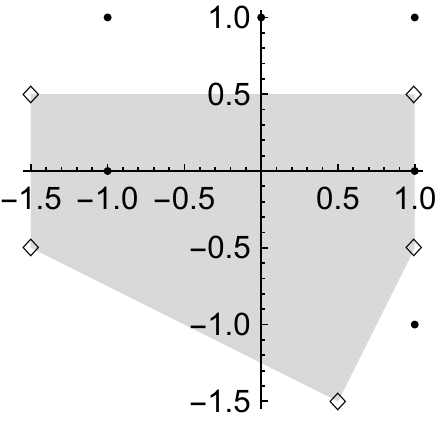}
      \caption{$\S_3$}
      \label{fig:pts3}
    \end{subfigure}
    \caption{Members of the collection $\SS'$, drawn with their individual $G$-invariant error sets.\label{fig:sprime}} 
\end{figure}

\begin{figure}
  \begin{subfigure}{.48\columnwidth}
  \centering
        \includegraphics[scale=.6]{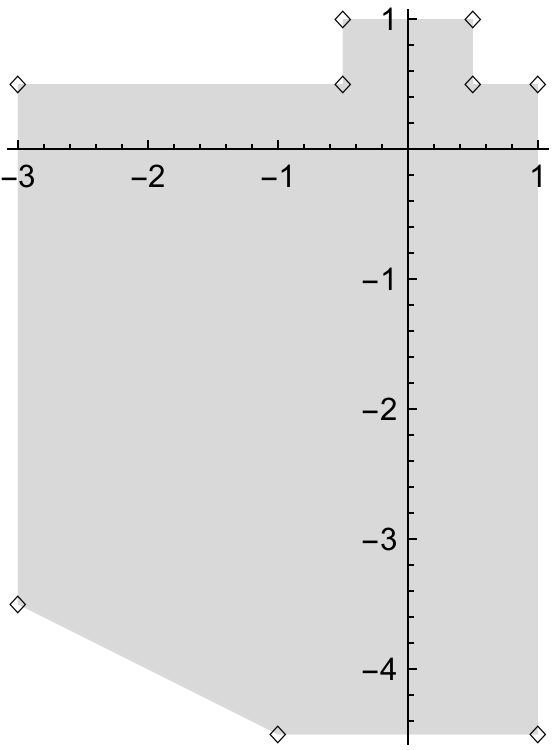}
\caption{The nearly minimal $G$-invariant error set for  $\mathbb{S}'\!=\!\{\mcal{S}_1,\mcal{S}_2,\mcal{S}_3\}$.
}
      \label{fig:jointG}
    \end{subfigure}
    \hfill
  \begin{subfigure}{.48\columnwidth}
  \centering
        \includegraphics[scale=.6]{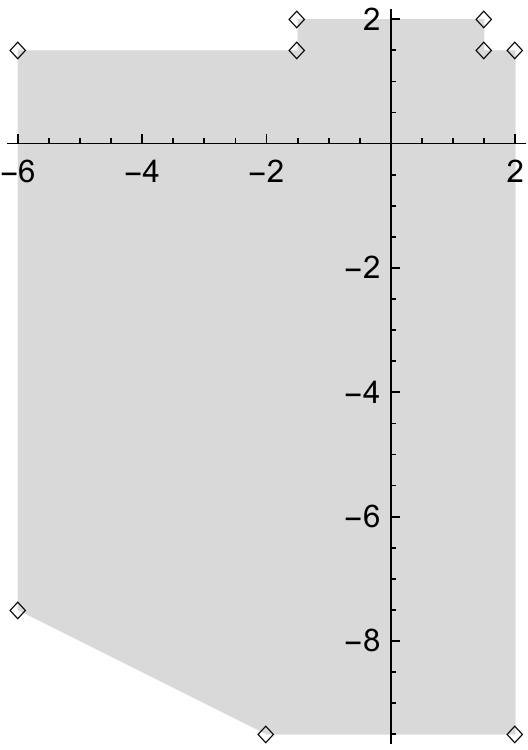}
\caption{The minimal $F$-invariant error set for  $\mathbb{S}'\!=\!\{\mcal{S}_1,\mcal{S}_2,\mcal{S}_3\}$, found after $362$ iterations.} \label{fig:finvset}
      \label{fig:jointF}
    \end{subfigure}
    \caption{Invariant sets corresponding to the collection $\SS'$.
    Note that the $G$-invariant error sets that correspond to the individual sets $\S_i$ (shown in \reffig{sprime}) are much smaller than the joint $G$-invariant error set. \label{fig:jointsets}}
    \vspace{-0.3cm}
\end{figure}


In our second example, we consider  
$\mathbb{S}' = \{ \mcal{S}_1,  \mcal{S}_2,  \mcal{S}_3 \}$, with 
\begin{align*}
\mcal{S}_1 &:= \{ (-1, -1),(0, -1),(1, -1),(1, 0),(1, 1),(0, 1), \\
           & \phantom{:=\{ } (-1, 1), (-1, 0) \}, \\
\mcal{S}_2 &:= \mcal{S}_1 \setminus \{(0,-1)\}, \\ 
\mcal{S}_3 &:= \mcal{S}_2 \setminus \{(-1,-1)\}.
\end{align*}
In words: the set $\mcal{S}_1$ is a collection of points that are placed equidistantly on a rectangle, and the set $\mcal{S}_2$ and $\mcal{S}_3$ respectively are created by removing one resp. two points from $\mcal{S}_1$; see \reffig{sprime}. This could correspond to a setting  in practice where the decision maker can implement points from $\mcal{S}_1$ most of the time, but once in a while one particular setpoint, (and sometimes even an additional particular setpoint) becomes temporarily infeasible. 

\subsection{Examples of $G$-Invariance}
The nearly minimal invariant error set, shown in \reffig{jointG}, has $9$ vertices and was found after $185$ iterations, with $\epsilon = 10^{-8}$ as rounding parameter.
From this figure, we see that the nearly minimal invariant error set corresponding to $\mathbb{S}'$ (the ``joint'' invariant error set) is significantly larger than the (nearly) minimal invariant error sets corresponding to singletons $\mathbb{S}''=\{\mcal{S}_i\}$ for all $i\in \{1,2,3\}$.


\subsection{Examples of $F$-Invariance}
For single-set collections $\SS=\{\S\}$, the minimal $F$-invariant affine region equals $\conv \S + Q$, where $Q$ is the minimal $G$-invariant error set. Hence, computing the $F$-invariant regions for the same single-set examples as used in the previous subsection again would be uninteresting. 
For collections $\SS$ with more than one member, the jointly $F$-invariant region is not equivalent to the jointly $G$-invariant set. For example, in \reffig{finvset} we show the jointly $F$-invariant region for the collection $\SS'$.  

\subsubsection*{An Example where \SS has Infinite Cardinality}
In this example we consider a triangle $\mcal{T}(h)$ whose height is parameterized by $h\in [0,h_{\max}]\subset \reals$: 
\[
  \mcal{T}(h) :=  \{ (x,y)\in \reals^2 : 0 \leq y \leq h, \frac{|x|}{y}\leq \tan \varphi \},
\]
where $\varphi$ is some fixed angle; 
see \reffig{Tset} for an illustration.
We define the collection of sets as $$\SS = \{\mcal{T}(h) \}_{0 \leq h \leq h_{\max}}.$$
Note that the collection $\SS$ has infinite cardinality.

The motivation for studying this particular example stems from an application in power systems: 
the triangular shape can be interpreted as the feasible set of a power converter for a photovoltaic panel, where the two dimensions correspond to active (along the $Y$ axis) and reactive power (along the $X$ axis), the angle $\varphi$  corresponds to the \emph{power factor}, and $h$ represents the amount of available power; see, e.g., \cite{commelec1}.

\begin{theorem} \label{thm:pv}
%
The set    $D := \mcal{T}(h_{\max})$ is the minimal $F$-invariant set with respect to $\SS = \{\mcal{T}(h) \}_{0 \leq h \leq h_{\max}}$.
%
\end{theorem}

\begin{figure*}[t]
  \vspace{-.5em}
\begin{subfigure}{.3\textwidth}
  \centering
  \includegraphics[scale=0.8]{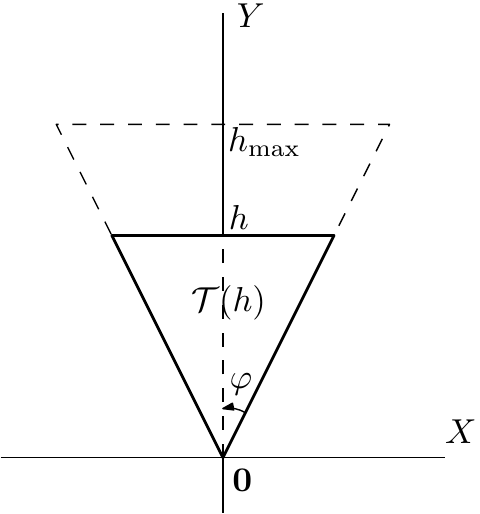}
  \caption{The parameterized collection of sets $\{ \mcal{T} (h): h\in [0,h_{\max}]\}$.\label{fig:Tset}}
\end{subfigure}
\hfill
\begin{subfigure}{.65\textwidth}
  \centering
  \includegraphics[scale=.8]{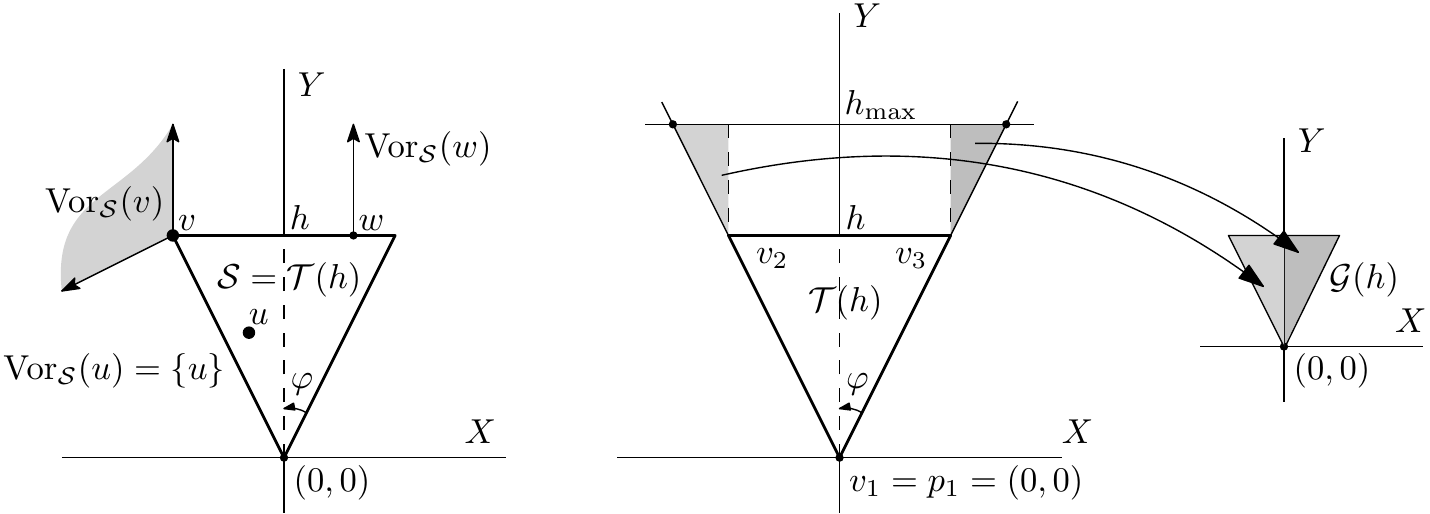}
  \caption{\emph{Left} -- The three Voronoi cell-types for the set \S, 1) a singleton for any interior point $u$, 2) a cone for any corner point $v$, and 3) a ray (normal to the corresponding facet) for any non-corner point $w$ on the boundary. \emph{Right} -- Construction of  $\mcal{G}(x)$.\label{fig:triangleconstr2}}
\end{subfigure}
\caption{An example where $\SS$ has infinite cardinality.}
\vspace{-1.3em}
\end{figure*}

\vspace{-0.3cm}

\begin{proof}
Observe that $\bigcap_{\S \in \SS} \S = \{0\}$ and 
\[
\bigcup_{\S \in \SS} \conv S = \bigcup_{\S \in \SS}  S = \bigcup_{0 \leq h \leq h_{\max}} \mcal{T}(h) = \mcal{T}(h_{\max}).
\]
Hence, it remains to show that $\mcal{T}(h_{\max})$ is $F$-invariant with respect to $\SS$. The minimality property then follows immediately from \reflem{union_min}. 
By \refprop{inv_mult_p}, it is enough to show that
\[
\mathfrak{p}_\SS(\mcal{T}(h_{\max})) \subseteq \mcal{T}(h_{\max}).
\]
In other words, we want to show that for any $\S \in \SS$,
\[
\S + \bigcup_{c \in \S} \vsect{\mcal{T}(h_{\max})} - c \subseteq \mcal{T}(h_{\max}).
\]

%
%
First, note that it is straightforward to characterize the different types of Voronoi cells of \S (see also \reffig{triangleconstr2}): 
for interior points, the Voronoi cell is the point itself;
for non-corner points on the boundary, the Voronoi cell is the outward-pointing ray that emanates from that point and is normal to the facet;
for corner points on the boundary, the Voronoi cell is a cone, namely the union of all rays that emanate from that corner point, whose directions vary (continuously) between the normals of the adjacent facts.

Now, it is not hard to see that for any $h\in[0,h_{\max}]$, we can construct $
\mcal{G}(h) := \bigcup_{c \in \mcal{T}(h)} \vsect{\mcal{T}(h_{\max})} - c
$  as shown in \reffig{triangleconstr2}. From this construction it then immediately follows that 
$\mcal{T}(h) + \mcal{G}(h) = \mcal{T}(h_{\max})$, which  
proves $F$-invariance of $\mcal{T}(h_{\max})$ with respect to $\mathbb{S}$. 
\end{proof}

\begin{corollary}
Let $D$ be as in \refthm{pv}. If $e_0 \in D$, then the accumulated error is bounded by
\[
\|e_n\| \leq \diam D = \max\l\{ P_{\max}/{\cos\varphi}, 2\, P_{\max}\tan \varphi \r\}
\]
for all $n \in \natnum$.
\end{corollary}
\begin{proof}
From \refprop{inv_bound_domain} we have that 
$\|e_n\| \leq \diam D$
 for all $n \in \natnum$.
Because $\mcal{T}(h_{\max})$ is an isosceles triangle, its diameter is either $h_{\max}/\cos \varphi$ (the length of one of its \emph{legs}) for $\varphi \leq\frac{\pi}{6}$ or $ 2\,h_{\max}\tan \varphi$ otherwise (the length of its \emph{base}).
\end{proof}

\vspace{-0.3cm}

\section{Conclusion}
We studied the problem of average point pursuit using the greedy algorithm under generalized conditions that include time-variability of the feasible sets and delay in information. We showed that, under certain conditions on the collection of the feasible sets, the greedy algorithm gives rise to bounded accumulated error dynamics, resulting in average tracking asymptotically. We illustrated the application of our results to general finite collections of discrete sets as well as to a particular case of a collection of infinite cardinality.



\section*{Acknowledgement}
We wish to thank Jean-Yves Le Boudec for helpful discussions about an earlier, related manuscript. 

\bibliography{refs,biblio}


\appendix
\section{Convex $G$-Invariant Sets}
\label{sec:convexsets}
Observe that the minimal $G$-invariant set is not necessarily convex. 
For various reasons (e.g., computational), convexity might be a desired property.
Hence, we next provide variants of Theorems \ref{theo:min_G} and \ref{theo:bound_G} on the existence of minimal \emph{convex} invariant sets and the way to compute them.

\begin{definition}[Convex set operator induced by a collection of sets] \label{def:g_conv}
Fix $d \in \natnum$.
For any 
set 
$\S \subseteq \reals^d$, define:
\[
\mathfrak{G}_\S(Q):=\conv{\left(\mathfrak{g}_\S(Q)\right)} = \conv{\left( \bigcup_{c \in \S } \vsect{(\conv{\S}+Q)} -c\right)}.
\]
For a collection $\SS$, let
\[
\mathfrak{G}_\SS(Q):=\bigcup_{\S \in \SS} \mathfrak{G}_\S(Q).
\]
\end{definition}

\begin{proposition} \label{prop:inv_conv}
A \emph{convex} set $A \subseteq \reals^d$ is invariant with respect to $\SS$ if and only if 
\[
A = \mathfrak{G}_\SS(A).
\]
\end{proposition}

\begin{proof}
By Lemma \ref{lem:A_in_g},
\[
A \subseteq \mathfrak{g}_\S(A) \subseteq \mathfrak{G}_\S(A).
\]
for every $\S \in \SS$. 
Hence, $A \subseteq \mathfrak{G}_\SS (A)$.

We next prove that $\mathfrak{G}_\SS(A) \subseteq A$ if and only if $A$ is invariant.

$(\Rightarrow)$ Assume $A \supseteq \mathfrak{G}_\SS (A)$. Then, also $A \supseteq \mathfrak{G}_{\S} (A) \supseteq \mathfrak{g}_{\S} (A)$ for all $\S \in \SS$. 
Thus, by Proposition \ref{prop:inv}, $A$ is invariant for all $\S \in \SS$, 
hence  invariant with respect to $\SS$.

$(\Leftarrow)$  Assume that $A$ is a convex invariant set. Also, assume by the way of contradiction that $A \not\supseteq \mathfrak{G}_\S(A)$ for some $\S \in \SS$. 
Namely, there exists $v \in \mathfrak{G}_\S(A)$ such that $v \notin A$. But every $v \in \mathfrak{G}_\S(A)$ can be written as $v = \sum_i \beta_i v_i$ with $v_i \in \mathfrak{g}_\S(A)$ and $\sum_i \beta_i = 1$, $\beta \geq 0$. By the invariance of $A$, we thus have that $v_i \in A$, but $v = \sum_i \beta_i v_i \notin A$, a contradiction to convexity of $A$. Therefore $A \supseteq \mathfrak{G}_\S(A)$ for all $\S \in \SS$, and hence $A \supseteq \mathfrak{G}_\SS(A)$ as required.
\end{proof}

\begin{theorem}[Minimal Convex Invariant Set] \label{theo:min_G_conv}
Let $Q \subseteq \reals^d$. The following statements hold:
\begin{enumerate}
    \item[(i)] The iterates
      \begin{align*}
  \mathfrak{G}^n_\SS (Q) &:= \mathfrak{G}_\SS (\mathfrak{G}^{n-1}_\SS (Q)) \quad \text{for } n\in \natnum, n \geq 1, \\
  \mathfrak{G}^0_\SS (Q) &:= \conv Q,
\end{align*}
are monotonic, in the sense that $\mathfrak{G}^n_\SS (Q) \subseteq \mathfrak{G}^{n'}_\SS (Q)$ for all $n \leq n'$, and the limit set 
\begin{equation} \label{eqn:conv_limit}
\mathfrak{G}^\infty_\SS (Q) := \lim_{n \rightarrow \infty} \mathfrak{G}^n_\SS (Q) = \bigcup_{n\geq0} \mathfrak{G}^n_\SS(Q)
\end{equation}
is the \emph{minimal convex invariant set} with respect to $\SS$ that contains $Q$.
\item[(ii)] Under the conditions of Theorem \ref{theo:bound_G}, $\mathfrak{G}^\infty_\SS (\{\boldsymbol{0}\})$ is a bounded set.
\end{enumerate}
\end{theorem}

\begin{proposition} \label{lem:conv_mono}
For all sets $A$ and $B$ with $A \subseteq B$,
\[
\conv{A} \subseteq \conv{B}.
\]
\end{proposition}

\begin{proof}[Proof of Theorem~\ref{theo:min_G_conv}]
First, by Lemma \ref{lem:g_prop_mult} (i) and the monotonicity of the convex hull (Proposition \ref{lem:conv_mono}), for all $A \subseteq B$  we have that $\mathfrak{G}_\SS(A) \subseteq \mathfrak{G}_\SS(B)$. Moreover,  by Lemma \ref{lem:A_in_g}, we have that $A \subseteq \mathfrak{G}_\S(A)$. Thus, applying the monotonicity property of $\mathfrak{G}_\S$ for this inclusion recursively, we obtain that existence of \eqref{eqn:conv_limit}. 

We next prove that $\mathfrak{G}^\infty_\SS (Q)$ is invariant. Similarly to the proof of Theorem \ref{theo:min_G}, we have that
\begin{eqnarray*}
\mathfrak{g}_\SS( \mathfrak{G}^\infty_\SS(A)) &=& \mathfrak{g}_\SS\left( \bigcup_{n \geq 0}  \mathfrak{G}^n_\SS(A) \right) 
= \bigcup_{n \geq 0} \mathfrak{g}_\SS\left(\mathfrak{G}^n_\SS(A) \right) \\
&\subseteq& \bigcup_{n \geq 0} \mathfrak{G}_\SS\left(\mathfrak{G}^n_\SS(A) \right) 
=\bigcup_{n \geq 1} \mathfrak{G}^n_\SS(A) \\
&=& A \cup \bigcup_{n \geq 1} \mathfrak{G}^n_\SS(A) = \mathfrak{G}^\infty_\SS(A),
\end{eqnarray*}
where the inclusion follows by the definition of the convex operator (Definition \ref{def:g_conv}).
To prove minimality, let $Q$ be any invariant convex set containing $A$. Then
\[
Q = \mathfrak{G}_\SS( Q ) \supseteq \mathfrak{G}_\SS( A ),
\]
where the first inclusion follows by Proposition \ref{prop:inv_conv} and the second inclusion follows by the monotonicity property. Applying these rules recursively, we obtain that
\[
Q \supseteq \mathfrak{G}_\SS^n( A ), \quad n \geq 0.
\]
This implies that
\[
Q \supseteq \bigcup_{n=0}^N \mathfrak{G}^n_\SS(A), \quad N \geq 0
\]
and hence
\[
Q \supseteq \mathfrak{G}^\infty_\SS(A).
\]
This completes the proof of part (i) of the Theorem.

Part (ii) of the Theorem follows from the proof of Theorem \ref{theo:bound_G} since it established the existence of a \emph{bounded convex} invariant set that contains the origin. As $\mathfrak{G}^\infty_\SS(\{\boldsymbol{0}\})$ is the minimal such set, it has to be bounded.
\end{proof}

\reffig{evo2} shows a numerical example of the convex $G$-iteration.

\begin{figure}
  \begin{subfigure}[b]{.24\textwidth}
  \centering
        \includegraphics[scale=.6,page=5]{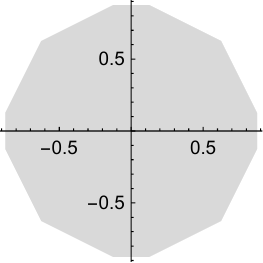}
        \caption{First iteration.}
        \label{fig:fulliter11}
    \end{subfigure}
    \hfill
  \begin{subfigure}[b]{.24\textwidth}
  \centering
        \includegraphics[scale=.6,page=6]{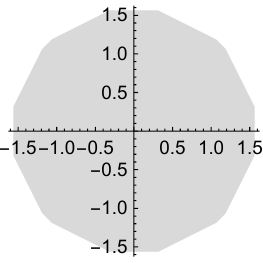}
        \caption{Second iteration.}
        \label{fig:fulliter22}
    \end{subfigure}
    \hfill
  \begin{subfigure}[b]{.24\textwidth}
  \centering
        \includegraphics[scale=.6,page=7]{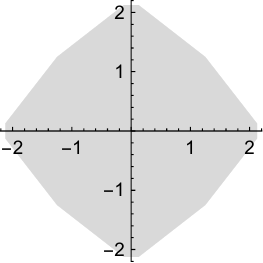}
        \caption{Third iteration.}
        \label{fig:fulliter33}
    \end{subfigure}
    \hfill
  \begin{subfigure}[b]{.24\textwidth}
  \centering
        \includegraphics[scale=.6,page=8]{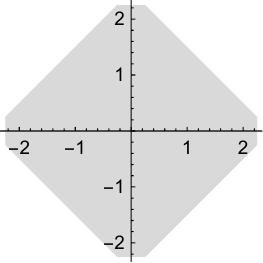}
        \caption{Fourth iteration.}
        \label{fig:fulliter44}
    \end{subfigure}
    \caption{Evolution of the iterate for the convex $G$-iteration ($\mathfrak{G}_\S$) for point set $\S :=\{ (\pm 2,0), (0,\pm 2), (\pm \tfrac12,\pm \tfrac12) \}.$ Convergence in the fourth iteration.\label{fig:evo2}}
\end{figure}


\section{Convex $F$-Invariant Sets}
Analogous to the convex $G$-iteration, 
we can define convexified counterparts of the set operators for the $F$-dynamics.

\begin{definition}[Set operators] \label{def:pconvex}
Fix $d \in \natnum$.
For any set $\S \subseteq \reals^d$, we define:
\[
\mathfrak{P}_\S(D):=\conv{\left(\mathfrak{p}_\S(D)\right)}.
\]
Also, for a collection \SS, let
\[
\mathfrak{P}_\SS (D):= \bigcup_{\S \in \SS} \mathfrak{P}_\S (D).
\]
Finally, we define the iterates of the above operators
 $\mathfrak{P}^n_\S(D)$ and $\mathfrak{P}^n_\SS(D)$, as before.
\end{definition}
Given our previous results, invariance, monotonicity, boundedness, and minimality can then be proven in a similar way also for the convex $F$ iteration.

\end{document}